\documentclass{article}
\usepackage{amsmath,amsfonts,amsthm,amssymb}
\usepackage{epsfig}
\usepackage{psfrag}
\usepackage[left=3cm, right=3cm, top=3cm, bottom=3.5cm]{geometry}
\usepackage{graphicx}
\usepackage{psfrag}
\usepackage{bbm}
\usepackage{color}
\usepackage{cancel}
\usepackage{psfrag}

\numberwithin{equation}{section}

\theoremstyle{plain}
\newtheorem{theorem}{Theorem}[section]

\newtheorem{lemma}[theorem]{Lemma}
\newtheorem{proposition}[theorem]{Proposition}
\newtheorem{corollary}[theorem]{Corollary}

\theoremstyle{remark}
\newtheorem*{remark}{Remark}

\theoremstyle{definition}

\newcommand{\R}{\mathbb{R}}
\newcommand{\N}{\mathbb{N}}

\def\eps{\epsilon}

\begin{document}

\title{On the existence of connecting orbits for critical values of the energy}
\author{{Giorgio
    Fusco,\footnote{Dipartimento di Matematica, Universit\`a dell'Aquila; e-mail:
      {\texttt{fusco@univaq.it}}}}\ \
Giovanni F. Gronchi,\footnote{Dipartimento di Matematica, Universit\`a
  di Pisa; e-mail:
  {\texttt{giovanni.federico.gronchi@unipi.it}}}\ \
Matteo Novaga\footnote{Dipartimento di Matematica, Universit\`a di Pisa; e-mail:
  {\texttt{matteo.novaga@unipi.it}}} }

\maketitle
\begin{abstract}
We consider an open connected set $\Omega$ and a smooth potential $U$
which is positive in $\Omega$ and vanishes on $\partial\Omega$. We
study the existence of orbits of the mechanical system
\[
\ddot{u}=U_x(u),
\]
that connect different components of $\partial\Omega$ and lie on the
zero level of the energy. We allow that $\partial\Omega$ contains a
finite number of critical points of $U$. The case of symmetric
potential is also considered.
\end{abstract}

\section{Introduction}
Let $U:\R^n\rightarrow\R$ be a
function of class $C^2$. We assume that $\Omega\subset\R^n$ is a
connected component of the set $\{x\in\R^n: U(x)>0\}$ and
that $\partial\Omega$ is compact and is the union of $N\geq 1$
distinct nonempty connected components $\Gamma_1,\ldots,\Gamma_N$. We
consider the following situations
\begin{description}
\item[H] $N\geq 2$ and, if $\Omega$ is unbounded, there is $r_0>0$ and
  a non-negative function $\sigma:[r_0,+\infty)\rightarrow\R$ such that
    $\int_{r_0}^{+\infty}\sigma(r)dr=+\infty$ and
  \begin{equation}
{\sqrt{U(x)}\geq\sigma(\vert x\vert),\;\; x\in\Omega,\;\;
  \vert x\vert\geq r_0.}
\label{sigmacond}
  \end{equation}
\item[H$_s$] $\Omega$ is bounded, the origin $0\in\R^n$ belongs to
  $\Omega$ and $U$ is invariant under the antipodal map
\[
U(-x)=U(x),\;\;x\in\Omega.
\]
\end{description}

\noindent Condition (\ref{sigmacond}) was first introduced in
\cite{monteil}. A sufficient condition for (\ref{sigmacond})
is that $\liminf_{|x|\to\infty} U(x) >0$.

\smallbreak
We study non constant solutions $u:(T_-,T_+)\rightarrow\Omega$, of the
equation
\begin{equation}\label{Newton}
\ddot{u}=U_x(u),\ \ \ U_x=\Bigl(\frac{\partial U}{\partial x}\Bigr)^T,
\end{equation}
that satisfy
\begin{equation}\label{lim}
\lim_{t\rightarrow T_\pm}d(u(t),\partial\Omega)=0,
\end{equation}
with $d$ the Euclidean distance, and lie on the energy surface
\begin{equation}\label{EnergConserv}
\frac{1}{2}\vert\dot{u}\vert^2-U(u)=0.
\end{equation}
 We allow that the boundary $\partial\Omega$ of $\Omega$ contains a
 finite set $P$ of critical
 points of $U$ and assume
\begin{description}
\item[H$_1$] If $\Gamma\in\{\Gamma_1,\ldots,\Gamma_N\}$ has positive
  diameter and $p\in P\cap\Gamma$ then $p$ is a hyperbolic critical
  point of $U$.
\end{description}
If $\Gamma$ has positive diameter, then hyperbolic critical points
$p\in\Gamma$ correspond to saddle-center equilibrium points in the
zero energy level of the Hamiltonian system associated to
(\ref{Newton}). These points are organizing centers of complex
dynamics, see \cite{salomao}.

\noindent Note that $\mathbf{H}_1$ does not exclude that some of the
$\Gamma_j$ reduce to a singleton, say $\{p\}$, for some $p\in P$.  In
this case nothing is required on the behavior of $U$ in a neighborhood
of $p$ aside from being $C^2$.

\noindent
A comment on $\mathbf{H}$ and $\mathbf{H}_s$ is in order. If $P$ is
nonempty $u\equiv p$ for $p\in P$ is a constant solution of
(\ref{Newton}) that satisfies (\ref{lim}) and (\ref{EnergConserv}). To
avoid trivial solutions of this kind we require $N\geq
  2$ in $\mathbf{H}$, and look for solutions that connect different
components of $\partial\Omega$. In $\mathbf{H}_s$ we do not exclude
that $\partial\Omega$ is connected ($N=1$) and avoid trivial solutions
by restricting to a symmetric context and to solutions that pass through
0.

\smallbreak
We prove the following results.
\begin{theorem}\label{main0}
Assume that $\mathbf{H}$ and $\mathbf{H}_1$ hold. Then for each
$\Gamma_-\in\{\Gamma_1,\ldots,\Gamma_N\}$ there exist
$\Gamma_+\in\{\Gamma_1,\ldots,\Gamma_N\}\setminus\{\Gamma_-\}$ and a
map $u^*:(T_-,T_+)\rightarrow\Omega$, with $-\infty\leq T_-<T_+\leq
+\infty$, that satisfies (\ref{Newton}), (\ref{EnergConserv}) and
\begin{equation}\label{d-0}
\lim_{t\rightarrow T_\pm}d(u^*(t),\Gamma_\pm)=0.
\end{equation}
Moreover, $T_->-\infty$ (resp. $T_+<+\infty$)  if and only if
$\Gamma_-$ (resp. $\Gamma_+$)
has positive diameter. If $T_->-\infty$ it results
\begin{equation}\label{x-}
\begin{split}
&\lim_{t\rightarrow T_-}u^*(t)= x_-,\\
&\lim_{t\rightarrow T_-}\dot{u}^*(t)=0,
\end{split}
\end{equation}
for some $x_-\in\Gamma_-\setminus P$. An analogous statement
holds if $T_+<+\infty$.
\end{theorem}
\begin{theorem}\label{main-0}
Assume that $\mathbf{H}_s$ and $\mathbf{H}_1$ hold. Then there exist
$\Gamma_+\in\{\Gamma_1,\ldots,\Gamma_N\}$ and a map
$u^*:(0,T_+)\rightarrow\Omega$, with $0<T_+\leq +\infty$, that satisfies
(\ref{Newton}), (\ref{EnergConserv}) and
\[
\lim_{t\rightarrow T_+}d(u^*(t),\Gamma_+)=0.
\]
Moreover, $T_+<+\infty$ if and only if $\Gamma_+$ has positive diameter. If
$T_+<+\infty$ it results
\[
\begin{split}
&\lim_{t\rightarrow T_+}u^*(t)=x_+,\\
&\lim_{t\rightarrow T_+}\dot{u}^*(t)=0,
\end{split}
\]
for some $x_+\in\Gamma_+\setminus P$.
\end{theorem}
We list a few straightforward consequences of Theorems \ref{main0} and
\ref{main-0}.

\begin{corollary}\label{Hetero}
  Theorem \ref{main0} implies that, if $\partial\Omega=P$, given $p_-\in
P$ there is $p_+\in P\setminus\{p_-\}$ and a heteroclinic connection
between $p_-$ and $p_+$, that is a solution $u^*:\R\rightarrow\R^n$ of
(\ref{Newton}) and (\ref{EnergConserv}) that satisfies
\[
\lim_{t\rightarrow\pm\infty}u^*(t)=p_\pm.
\]
\end{corollary}
The problem of the existence of heteroclinic connections between two
isolated zeros $p_\pm$ of a non-negative potential has been recently
reconsidered by several authors. In \cite{af} existence was established
under a mild monotonicity condition on $U$ near $p_\pm$. This
condition was removed in \cite{sourdis}, see also \cite{A}. The most
general results, equivalent to the consequence of Theorem \ref{main0}
discussed in Section~\ref{s:heteroclinic}, were recently obtained in \cite{monteil} and in
\cite{ZS}, see also \cite{braides}.
All these papers establish existence by a variational
approach. In \cite{af}, \cite{sourdis} and \cite{A} by minimizing the
action functional, and in \cite{monteil} and
\cite{ZS} by minimizing the Jacobi functional.

\begin{corollary}\label{Homo}
  Theorem \ref{main0} implies that,
    if $\Gamma_-=\{p\}$ for some $p\in P$ and the elements of
  $\{\Gamma_1,\ldots,\Gamma_N\}\setminus\{\Gamma_-\}$ have all positive
diameter, there exists a nontrivial orbit homoclinic to $p$ that satisfies (\ref{Newton}), (\ref{EnergConserv}).
\end{corollary}
\begin{proof}
Let $v^*:\R\rightarrow\Omega\cup\{x_+\}$  be the extension
defined by
\[
v^*(T_++t)=u^*(T_+-t),\;\;t\in(0,+\infty),\;\;v^*(T_+)=x_+,
\]
of the solution $u^*:(-\infty,T_+)\rightarrow\Omega$ given by
Theorem~\ref{main0}.
The map $v^*$ so defined is a smooth non-constant solution of
(\ref{Newton}) that satisfies
\[
\lim_{t\rightarrow\pm\infty}v^*(t)=p.
\]
\end{proof}
\begin{corollary}\label{Per}
Theorem \ref{main0} implies that, if all the sets
$\Gamma_1,\ldots,\Gamma_N$ have positive diameter, given
$\Gamma_-\in\{\Gamma_1,\ldots,\Gamma_N\}$, there exist
$\Gamma_+\in\{\Gamma_1,\ldots,\Gamma_N\}\setminus\{\Gamma_-\}$ and a
periodic solution $v^*:\R\rightarrow\Omega$ of (\ref{Newton}) and
(\ref{EnergConserv}) that oscillates between $\Gamma_-$ and
$\Gamma_+$. This solution has period $T=2(T_+-T_-)$.
\end{corollary}
\begin{proof}
The solution $v^*$ is the $T$-periodic extension of the map
$w^*:[T_-,2T_+-T_-]\rightarrow \Omega$ defined by $w^*(t)=u^*(t)$ for
$t\in(T_-,T_+)$, where $u^*$ is given by Theorem \ref{main0}, and
\[
\begin{split}
&w^*(T_\pm)=x_\pm,\\
&w^*(T_++t)=u^*(T_+-t),\quad t\in (0,T_+-T_-].
\end{split}
\]
\end{proof}

The problem of existence of heteroclinic, homoclinic and periodic
solutions of (\ref{Newton}), in a context similar to the one considered
here, was already discussed in \cite{A} where $\partial\Omega$ is
allowed to include continua of critical points. Our result concerning
periodic solutions extends a corresponding result in \cite{A} where
existence was established under the assumption that $P=\emptyset$.

The following result is a direct consequence of Theorem \ref{main-0}.
\begin{corollary}\label{caseN-1}
Theorem \ref{main-0} implies that, if all the sets $\Gamma_1,\ldots,\Gamma_N$ have positive diameter,
  there exists $\Gamma_+\in\{\Gamma_1,\ldots,\Gamma_N\}$ and a
  periodic solution $v^*:\R\rightarrow\Omega$ of (\ref{Newton}) and
  (\ref{EnergConserv}) that satisfies
  \[
v^*(-t) = -v^*(t), \quad t\in\R. 
  \]
  This solution has period
  $T=4T_+$, with $T_+$.
\end{corollary}
\begin{proof}
The solution $v^*$ is the $T$-periodic extension of the map
$w^*:[-2T_+,2T_+]\rightarrow \Omega$ defined by $w^*(t) = u^*(t)$ for
$t\in(0,T_+)$, where $u^*$ is given by Theorem \ref{main-0}, and by
\[
\begin{split}
  &w^*(t) = -w^*(-t),\hskip 2cm t\in (-T_+,0),\\
  &w^*(0) = 0, \quad w^*(\pm T_+) = \pm x_+, \\
  &w^*(T_++t) = w^*(T_+-t), \hskip 0.8cm t\in(0,T_+],\\
  &w^*(-T_++t) = w^*(-T_+-t), \hskip 0.3cm t\in[-T_+,0).\\
\end{split}
\]
In particular the solution oscillates between $x_+$ and $-x_+$ and
this is true also when $\partial\Omega$ is connected ($N=1$).
\end{proof}

\section{Proof of Theorems \ref{main0} and \ref{main-0}}

We recall a classical result.
\begin{lemma}\label{JacobiL}
Let $G:\R^n\rightarrow\R$  be a smooth bounded and
non-negative potential, $I=(a,b)$ a bounded interval.
Define the Jacobi functional
\[
\mathcal{J}_G(q, I)=\sqrt{2}\int_I\sqrt{G(q(t))}\vert\dot{q}(t)\vert dt
\]
and the action functional
\[
\mathcal{A}_G(q,I) =
\int_I\Big(\frac{1}{2}\vert\dot{q}(t)\vert^2+G(q(t))\Big)dt.
\]
Then
\begin{enumerate}
\item
\[
\mathcal{J}_G(q, I)
\leq
\mathcal{A}_G(q,I),
\quad q\in W^{1,2}(I;\R^n)
\]
 with equality sign if and only if
\[
\frac{1}{2}\vert\dot{q}(t)\vert^2-G(q(t))=0,\;t\in I.
\]
\item
\[
\min_{q\in\mathcal{Q}}\mathcal{J}_G(q, I)=\min_{q\in\mathcal{Q}}\mathcal{A}_G(q,I),
\]
where
\[
\mathcal{Q}=\{q\in W^{1,2}(I;\R^n):q(a)=q_a, q(b)=q_b\}.
\]
\end{enumerate}
\end{lemma}
\noindent When $G=U$ we shall simply write $\mathcal{J}, \mathcal{A}$ for
$\mathcal{J}_U, \mathcal{A}_U$.

\smallbreak We now start the proof of Theorem~\ref{main0}. Choose
$\Gamma_-\in\{\Gamma_1,\ldots,\Gamma_N\}$ and set
\[
d=\min\{\vert x-y\vert: x\in\Gamma_-, y\in\partial\Omega\setminus\Gamma_-\}.
\]
For small $\delta\in(0,d)$ let $O_\delta=\{x\in\Omega:
d(x,\Gamma_-)<\delta\}$ and let $U_0=\frac{1}{2}\min_{x\in\partial
  O_\delta\cap\Omega}U(x)$. We note that $U_0>0$ and define the
admissible set
\begin{equation}\label{Ad}
\begin{split}
\mathcal{U}=\bigl\{u\in W^{1,2}((T_-^u,T_+^u);\R^n):
        {-\infty< T_-^u < T_+^u< +\infty,}\hskip 2.2cm\\
        u((T_-^u,T_+^u))\subset\Omega,\ U(u(0))=U_0,\
        u(T_-^u)\in\Gamma_-,\
        u(T_+^u)\in\partial\Omega\setminus\Gamma_-\bigr\}.
\end{split}
\end{equation}
We determine the map $u^*$ in Theorem \ref{main0} as the limit of a
minimizing sequence $\{u_j\}\subset\mathcal{U}$ of the action
functional
\[
\mathcal{A}(u,(T_-^u,T_+^u)) =
\int_{T_-^u}^{T_+^u}\Big(\frac{1}{2}\vert\dot{u}(t)\vert^2+U(u(t))\Big)dt,
\]

\noindent Note that in the definition of $\mathcal{U}$ the times
$T_-^u$ and $T_+^u$ are not fixed but, in general,
change with $u$. Note also that the condition
$U(u(0))=U_0$ in (\ref{Ad}) is
a normalization which can always be imposed by a translation of time
and has the scope of eliminating the loss of compactness due to
translation invariance.
Let $\bar{x}_-\in\Gamma_-$ and
$\bar{x}_+\in\partial\Omega\setminus\Gamma_-$ be such that $\vert
\bar{x}_+-\bar{x}_-\vert=d$ and set
\[
\tilde{u}(t)=(1-(t+\tau))\bar{x}_-+(t+\tau)\bar{x}_+,\;\;t\in[-\tau,1-\tau],
\]
where $\tau\in(0,1)$ is chosen so that $U(\tilde{u}(0))=U_0$. Then
$\tilde{u}\in\mathcal{U}$, $T_-^{\tilde{u}}=-\tau$,
$T_+^{\tilde{u}}=1-\tau$ and
\[
\mathcal{A}(\tilde{u},(-\tau,1-\tau))=a<+\infty.
\]
Next we show that there are constants $M>0$ and $T_0>0$ such that
each $u\in\mathcal{U}$ with
\begin{equation}\label{Bound-E}
  \mathcal{A}(u,(T_-^u,T_+^u))\leq a,
\end{equation}
satisfies
\begin{equation}\label{Bound}
\begin{split}
&\|u\|_{L^\infty((T_-^u,T_+^u);\R^n)}\leq M,\\
& T_-^u\leq-T_0<T_0\leq T_+^u.
\end{split}
\end{equation}

\noindent
The $L^\infty$ bound on $u$ follows from $\mathbf{H}$ and from
Lemma~\ref{JacobiL}, in fact, if $\Omega$ is unbounded, $\vert
u(\bar{t})\vert=M$ for some $\bar{t}\in(T_-^u,T_+^u)$ implies
\[
a\geq\mathcal{A}(u,(T_-^u,\bar{t})) \geq
\int_{T_-^u}^{\bar{t}}\sqrt{2U(u(t))}\vert\dot{u}(t)\vert
dt\geq\sqrt{2}\int_{r_0}^M\sigma(s)ds.
\]
The existence of $T_0$ follows from
\[
\frac{d_1^2}{|T_-^u|}\leq\int_{T_-^u}^0\vert\dot{u}(t)\vert^2dt\leq 2a,\qquad
\frac{d_1^2}{T_+^u}\leq\int_0^{T_+^u}\vert\dot{u}(t)\vert^2dt\leq 2a,
\]
where $d_1=d(\partial\Omega,\{x: U(x)>U_0\})$.

Let $\{u_j\}\subset\mathcal{U}$ be a minimizing sequence
\begin{equation}
\lim_{j\rightarrow+\infty}\mathcal{A}(u_j,(T_-^{u_j},T_+^{u_j})) =
\inf_{u\in\mathcal{U}}\mathcal{A}(u,(T_-^u,T_+^u)) := a_0\leq a.
\label{a0bound}
\end{equation}
We can assume that each $u_j$ satisfies (\ref{Bound-E}) and
(\ref{Bound}). By considering a subsequence, that we still denote by $\{u_j\}$,
we can also assume that there exist $T_-^\infty$, $T_+^\infty$ with $-\infty\leq
T_-^\infty\leq-T_0<T_0\leq T_+^\infty \leq +\infty$ and a continuous
map $u^*:(T_-^\infty,T_+^\infty)\rightarrow\R^n$ such that
\begin{equation}\label{uStar}
\begin{split}
&\lim_{j\rightarrow+\infty}T_\pm^{u_j}=T_\pm^\infty,\\
&\lim_{j\rightarrow+\infty}u_j(t)=u^*(t),\;\;t\in(T_-^\infty,T_+^\infty),
\end{split}
\end{equation}
and in the last limit the convergence is uniform on bounded
intervals. This follows from (\ref{Bound}) which
implies that the sequence $\{u_j\}$ is equi-bounded and from (\ref{Bound-E}) which implies
\begin{equation}\label{Holder}
  \vert u_j(t_1)-u_j(t_2)\vert\leq
  \left\vert\int_{t_1}^{t_2}\vert\dot{u}_j(t)\vert dt\right\vert\leq\sqrt{a}\vert t_1-t_2\vert^\frac{1}{2},
\end{equation}
so that the sequence is also equi-continuous.

By passing to a further subsequence we can also assume
that $u_j\rightharpoonup u^*$ in
$W^{1,2}((T_1,T_2);\R^n)$ for each $T_1$, $T_2$ with
$T_-^\infty<T_1<T_2<T_+^\infty$.  This follows from (\ref{Bound-E}),
which implies
\[
\frac{1}{2}\int_{T_-^{u_j}}^{T_+^{u_j}}\vert\dot{u}_j\vert^2dt \leq
\mathcal{A}(u_j,(T_-^{u_j},T_+^{u_j}))\leq a,
\]
and from the fact that each map $u_j$ satisfies (\ref{Bound}) and
therefore is bounded in $L^2((T_-^{u_j},T_+^{u_j});\R^n)$.

\noindent We also have
\begin{equation}\label{IV}
\mathcal{A}(u^*,(T_-^\infty,T_+^\infty))\leq a_0.
\end{equation}

\noindent Indeed, from the lower semicontinuity of the norm, for each
$T_1$, $T_2$ with $T_-^\infty<T_1<T_2<T_+^\infty$ we have
\[
\int_{T_1}^{T_2}\vert\dot{u}^*\vert^2dt \leq
\liminf_{j\rightarrow+\infty}\int_{T_1}^{T_2}\vert\dot{u}_j\vert^2dt.
\]
This and the fact that $u_j$ converges to $u^*$ uniformly in
$[T_1,T_2]$
imply
\[
\mathcal{A}(u^*,(T_1,T_2)) \leq
\liminf_{j\rightarrow+\infty}\mathcal{A}(u_j,(T_1,T_2)) \leq
\liminf_{j\rightarrow+\infty}\mathcal{A}(u_j,(T_-^{u_j},T_+^{u_j}))=a_0.
\]
Since this is valid for each $T_-^\infty<T_1<T_2<T_+^\infty$ the claim
(\ref{IV}) follows.

\begin{lemma}\label{timelemma}
Define
 $T_-^\infty\leq T_-\leq-T_0<T_0\leq T_+\leq
T_+^\infty$ by setting
\[
\begin{split}
& T_-=\inf\{t\in(T_-^\infty,0]:u^*((t,0])\subset\Omega\}\\
&T_+=\sup\{t\in(0,T_+^\infty):u^*([0,t))\subset\Omega\}.
\end{split}
\]
Then
\begin{enumerate}
\item
\begin{equation}\label{IV0}
\mathcal{A}(u^*,(T_-,T_+))=a_0.
\end{equation}
\item $T_+<+\infty$ implies $\lim_{t\rightarrow T_+}u^*(t)=x_+$ for some $x_+\in\Gamma_+$ and
$\Gamma_+\in\{\Gamma_1,\ldots,\Gamma_N\}\setminus\{\Gamma_-\}$.
\item $T_+=+\infty$ implies
\begin{equation}\label{LimD0}
\lim_{t\rightarrow+\infty}d(u^*(t),\Gamma_+)=0,
\end{equation}
 for some $\Gamma_+\in\{\Gamma_1,\ldots,\Gamma_N\}\setminus\{\Gamma_-\}$.
\end{enumerate}
Corresponding statements apply to $T_-$.
\end{lemma}
\begin{proof}
We first prove $(ii)$, $(iii)$. If $T_+<+\infty$ the existence of 
$\lim_{t\rightarrow T_+}u^*(t)$ follows from (\ref{Holder}) which
implies that $u^*$ is a $C^{0,\frac{1}{2}}$ map. The limit $x_+$
belongs to $\partial\Omega$ and therefore to $\Gamma_+$ for some
$\Gamma_+\in\{\Gamma_1,\ldots,\Gamma_N\}$.
Indeed,
$x_+\not\in\partial\Omega$ would imply the existence of $\tau>0$ such
that, for $j$ large enough,
\[
d(u_j([T_+,T_++\tau]),\partial\Omega)\geq \frac{1}{2}d(x_+,\partial\Omega),
\]
in contradiction with the definition of $T_+$.  If $T_+=+\infty$ and
(iii) does not hold there is $\delta>0$ and a diverging sequence
$\{t_j\}$ such that
\[
d(u^*(t_j),\partial\Omega)\geq\delta.
\]
Set $U_m=\min_{d(x,\partial\Omega)=\delta}U(x)>0$. From the uniform
continuity of $U$ in $\{\vert x\vert\leq M\}$ ($M$ as in (\ref{Bound})) it
follows that there is $l>0$ such that
\[
\vert
U(x_1)-U(x_2)\vert\leq\frac{1}{2}U_m,\;\;\text{ for }\;\vert
x_1-x_2\vert\leq l, \ x_1,x_2\in \{\vert x\vert\leq M\}.
\]
This and $u^*\in C^{0,\frac{1}{2}}$ imply
\[
U(u^*(t))\geq\frac{1}{2}U_m,\;\;t\in I_j =
\Bigl(t_j-\frac{l^2}{a},t_j+\frac{l^2}{a}\Bigr),
\]
and, by passing to a subsequence, we can assume that the intervals
$I_j$ are disjoint. Therefore for each $T>0$ we have
\[
\sum_{t_j\leq T}\frac{l^2U_m}{a}\leq\int_0^TU(u^*(t))dt\leq a_0,
\]
which is impossible for $T$ large. This establishes (\ref{LimD0}) for
some $\Gamma_+\in\{\Gamma_1,\ldots,\Gamma_N\}$.  It remains to show
that $\Gamma_+\neq\Gamma_-$. This is a consequence of the minimizing
character of $\{u_j\}$. Indeed, $\Gamma_+=\Gamma_-$ would imply the
existence of a constant $c>0$ such that
$\lim_{j\to\infty}\mathcal{A}(u_j,(T_-^{u_j},T_+^{u_j}))\geq a_0+c$.

Now we prove $(i)$.  $T_+- T_-<+\infty$, implies that $u^*$ is an
element of $\mathcal{U}$ with $T_\pm^{u^*}= T_\pm$. It follows that
$\mathcal{A}(u^*,(T_-,T_+))\geq a_0$, which together with (\ref{IV})
imply (\ref{IV0}).  Assume now $T_+- T_-=+\infty$. If $T_+=+\infty$,
(\ref{LimD0}) implies that, given a small number $\epsilon>0$, there
are $t_\epsilon$ and $\bar{x}_\epsilon\in\partial\Omega$ such that
${\vert u^*(t_\epsilon)-\bar{x}_\epsilon\vert=\epsilon}$ and the
segment joining $u^*(t_\epsilon)$ to $\bar{x}_\epsilon$ belongs
to $\overline\Omega$.  Set
\[
v_\epsilon(t) =
(1-(t-t_\epsilon))u^*(t_\epsilon) +
(t-t_\epsilon)\bar{x}_\epsilon,
\;\;t\in(t_\epsilon,t_\epsilon+1].
\]
From the uniform continuity of $U$ there is $\eta_\epsilon>0$, $\lim_{\epsilon\rightarrow 0}\eta_\epsilon=0$, such that $U(v_\epsilon(t))\leq\eta_\epsilon$, for $t\in[t_\epsilon,t_\epsilon+1]$. Therefore we have
\[
 \mathcal{A}(v_\epsilon,(t_\epsilon,t_\epsilon+1))
 \leq\frac{1}{2}\epsilon^2+\eta_\epsilon.
\]
If $T_->-\infty$ the map $u_\epsilon =
\mathbbm{1}_{[T_-,t_\epsilon]}u^* +
\mathbbm{1}_{(t_\epsilon,t_\epsilon+1]}v_\epsilon$
  belongs to $\mathcal{U}$ and it results
\[
a_0\leq\mathcal{A}(u_\epsilon,(T_-,t_\epsilon+1)) =
\mathcal{A}(u^*,(T_-,t_\epsilon)) +
\mathcal{A}(v_\epsilon,(t_\epsilon,t_\epsilon+1))
\leq\mathcal{A}(u^*,(T_-,T_+))+\frac{1}{2}\epsilon^2+\eta_\epsilon.
\]
Since this is valid for all small $\epsilon>0$ we get
\[
a_0\leq\mathcal{A}(u^*,(T_-,T_+)),
\]
that together with (\ref{IV})
establishes (\ref{IV0}) if $T_->-\infty$ and $T_+=+\infty$. The
discussion of the other cases where $T_+-T_- =+\infty$ is similar.
\end{proof}

We observe that there are cases with $T_+<T_+^\infty$ and/or
  $T_->T_-^\infty$, see Remark~\ref{tempidiversi}.

\begin{lemma}\label{FirstInt}
The map $u^*$ satisfies (\ref{Newton}) and (\ref{EnergConserv}) in
$(T_-,T_+)$.
\end{lemma}
\begin{proof}
  1. We first show that
%
for each $T_1$, $T_2$ with $T_-<T_1<T_2<T_+$ we have
\begin{equation}\label{Tmin}
  \mathcal{A}(u^*,(T_1,T_2))=\inf_{v\in\mathcal{V}}\mathcal{A}(v,(T_1,T_2)),\\
\end{equation}
where
\[
\mathcal{V}=\{v\in W^{1,2}((T_1,T_2);\R^n): v(T_i)=u^*(T_i),i=1,2;\,
  v([T_1,T_2])\subset\Omega\}.
\]
Suppose instead that there are $\eta>0$ and $v\in\mathcal{V}$ such that
\[
\mathcal{A}(v,(T_1,T_2))=\mathcal{A}(u^*,(T_1,T_2))-\eta.
\]
Set $w_j:(T_-^{u_j},T_+^{u_j})\rightarrow\Omega$ defined by
\[
w_j(t)=\left\{\begin{array}{l}
u_j(t),\;\;t\in(T_-^{u_j},T_1]\cup[T_2,T_+^{u_j}),\\
\stackrel{}{
v(t)+\displaystyle\frac{T_2-t}{T_2-T_1}\delta_{1j}+\frac{t-T_1}{T_2-T_1}\delta_{2j},
\;\;t\in(T_1,T_2),}
\end{array}\right.
\]
where $\delta_{ij}=u_j(T_i)-u^*(T_i)$, $i=1,2$, with $u_j$ as in (\ref{a0bound}).
Define $v_j:[T_-^{v_j},T_+^{v_j}]\rightarrow\R^n$ by
\[
v_j(t) = w_j(t - \tau_j),
\]
where $\tau_j$ is such that $U(v_j(0))=U_0$, as in
(\ref{Ad}). Note that
\begin{equation}
\mathcal{A}(v_j,(T_-^{v_j},T_+^{v_j}))= \mathcal{A}(w_j,(T_-^{u_j},T_+^{u_j})).
\label{equalaction}
\end{equation}
From (\ref{uStar}) we have
  $\lim_{j\to\infty}\delta_{ij}= 0, i=1,2$, so that
\[
\lim_{j\rightarrow+\infty}\mathcal{A}(w_j,(T_1,T_2))=
\mathcal{A}(v,(T_1,T_2))
=
\mathcal{A}(u^*,(T_1,T_2))-\eta
\leq\liminf_{j\rightarrow+\infty}\mathcal{A}(u_j,(T_1,T_2))-\eta.
\]
Therefore we have
\[
\begin{split}
  &\liminf_{j\rightarrow+\infty}\mathcal{A}(w_j,(T_-^{u_j},T_+^{u_j}))=
  \lim_{j\rightarrow+\infty}\mathcal{A}(w_j,(T_1,T_2))+
\liminf_{j\rightarrow+\infty}\mathcal{A}(u_j,(T_+^{u_j},T_1)\cup(T_2,T_+^{u_j}))\\
&\leq\liminf_{j\rightarrow+\infty}\mathcal{A}(u_j,(T_1,T_2))-\eta+
\liminf_{j\rightarrow+\infty}\mathcal{A}(u_j,(T_+^{u_j},T_1)\cup(T_2,T_+^{u_j}))\leq
a_0-\eta,
\end{split}
\]
that, given (\ref{equalaction}), is in contradiction with the
minimizing character of the sequence $\{u_j\}$.

\noindent
The fact that $u^*$ satisfies (\ref{Newton}) follows from (\ref{Tmin})
and regularity theory, see \cite{BGH}.
To show that $u^*$ satisfies (\ref{EnergConserv}) we
distinguish the case $T_+-T_-<+\infty$ from the case
$T_+-T_-=+\infty$.

\noindent
2. $T_+-T_-<+\infty$.
Given $t_0, t_1$ with $T_-<t_0<t_1<T_+$, let
$\phi:[t_0,t_1+\tau]\rightarrow[t_0,t_1]$ be linear, with $|\tau|$
small, and let $\psi:[t_0,t_1]\rightarrow[t_0,t_1+\tau]$ be the
inverse of $\phi$. Define $u_\tau:[T_-,T_++\tau]\rightarrow\R^n$
by setting
\begin{equation}\label{ugamma}
u_\tau(t)=\left\{\begin{array}{l}
u^*(t),\;\;t\in[T_-,t_0],\\
u^*(\phi(t)),\;\;t\in[t_0,t_1+\tau],\\
u^*(t-\tau),\;\;t\in(t_1+\tau,T_++\tau)]
\end{array}\right.
\end{equation}
Note that $u_\tau\in\mathcal{U}$ with $T_-^{u_\tau}=T_-$ and
$T_+^{u_\tau}=T_++\tau$.  Since $u^*$ is a
minimizer we have
\begin{equation}\label{DerA}
\frac{d}{d\tau}\mathcal{A}(u_\tau,(T_-^{u_\tau},T_+^{u_\tau}))\vert_{\tau=0}=0.
\end{equation}
From (\ref{ugamma}), using also the change of variables $t=\psi(s)$, it
follows
\[
\begin{split}
&\mathcal{A}(u_\tau,(T_-^{u_\tau},T_+^{u_\tau}))-\mathcal{A}(u^*,(T_-,T_+))\\
& = \int_{t_0}^{t_1+\tau}
  \Big(\frac{\dot{\phi}^2(t)}{2}\vert\dot{u}^*(\phi(t))\vert^2 +
  U(u^*(\phi(t)))\Big)dt -\int_{t_0}^{t_1}
  \Big(\frac{1}{2}\vert\dot{u}^*(t)\vert^2+U(u^*(t))\Big)dt\\
& = \int_{t_0}^{t_1} \Big(\frac{1 -
    \dot{\psi}(t)}{2\dot{\psi}(t)}\vert\dot{u}^*(t)\vert^2 +
  (\dot{\psi}(t)-1)U(u^*(t))\Big)dt\\
& = \int_{t_0}^{t_1} \Big(\frac{-\frac{\tau}{t_1-t_0}}{2(1 +
    \frac{\tau}{t_1-t_0})}\vert\dot{u}^*(t)\vert^2 +
  \frac{\tau}{t_1-t_0}U(u^*(t))\Big)dt\\
&= - \frac{\tau}{t_1-t_0} \int_{t_0}^{t_1}\Big(
  \frac{\vert\dot{u}^*(t)\vert^2} {2(1 + \frac{\tau}{t_1-t_0})} -
  U(u^*(t))\Big)dt .
\end{split}
\]
This and (\ref{DerA}) imply
\begin{equation}\label{Int0}
\int_{t_0}^{t_1}\Big(\frac{1}{2}\vert\dot{u}^*(t)\vert^2-U(u^*(t))\Big)dt=0.
\end{equation}
Since this holds for all $t_0,t_1$, with $T_-<t_0<t_1<T_+$, then (\ref{EnergConserv}) follows.

\noindent
3. $T_+-T_-=+\infty$. We only consider the case
$T_+=+\infty$. The discussion of the other cases is similar.
Let $T\in(T_-,+\infty)$, let $T_-<t_0<t_1<T$ and let
$\phi:[t_0,T]\rightarrow[t_0,T]$ be linear in the intervals
$[t_0,t_1+\tau]$, $[t_1+\tau,T]$, with $|\tau|$ small, and such that
$\phi([t_0,t_1+\tau])=[t_0,t_1]$.
Define $u_\tau:(T_-,+\infty)\rightarrow\R^n$ by setting
\[
u_\tau(t)=\left\{\begin{array}{l}
u^*(t),\;\;t\in(T_-,t_0]\cup[T,+\infty)\\
u^*(\phi(t)),\;\;t\in[t_0,T].
\end{array}\right.
\]
We have
\[
\begin{split}
&\mathcal{A}(u_\tau,(T_-,T))-\mathcal{A}(u^*,(T_-,T))\\
& = \int_{t_0}^{t_1}\Big(\frac{-\frac{\tau}{t_1-t_0}}{2(1 +
    \frac{\tau}{t_1-t_0})}\vert\dot{u}^*(t)\vert^2 +
  \frac{\tau}{t_1-t_0}U(u^*(t))\Big)dt +
  \int_{t_1}^{T}\Big(\frac{\frac{\tau}{T-t_1}}{2(1 +
    \frac{\tau}{T-t_1})}\vert \dot{u}^*(t)\vert^2
  -\frac{\tau}{T-t_1}U(u^*(t))\Big)dt.
\end{split}
\]
Since $u^*$ restricted to the interval $[t_0,T]$ is a minimizer of
(\ref{Tmin}), by differentiating with respect to $\tau$ and setting
$\tau=0$ we obtain
\[
-\frac{1}{t_1-t_0}\int_{t_0}^{t_1}\Big(\frac{1}{2}\vert\dot{u}^*(t)\vert^2
-U(u^*(t))\Big)dt
+\frac{1}{T-t_1}\int_{t_1}^{T}\Big(\frac{1}{2}\vert\dot{u}^*(t)\vert^2
-U(u^*(t))\Big)dt=0.
\]
From (\ref{IV}) it follows that the second term in this expression
converges to zero when $T\rightarrow+\infty$. Therefore, after taking
the limit for $T\rightarrow+\infty$, we get back to (\ref{Int0}) and,
as before, we conclude that (\ref{EnergConserv}) holds.
\end{proof}

\begin{lemma}\label{limd0}
Assume that $\lim_{t\rightarrow T_+}u^*(t)=p\in P$. Then
\[
T_+=+\infty.
\]
\end{lemma}
\begin{proof}
Since $U$ is of class $C^2$ and $p$ is a critical point of $U$ there are constants $c>0$ and $\rho>0$ such that
\[U(x)\leq c\vert x-p\vert^2,\;\;x\in B_\rho(p)\cap\Omega.\]
Fix $t_\rho$ so that $u^*(t)\in B_\rho(p)\cap\Omega$ for $t\geq
t_\rho$. Then $T_+=+\infty$ follows from (\ref{EnergConserv}) and
\[
\frac{d}{dt}\vert u^*-p\vert \geq -\vert\dot{u}^*\vert=-\sqrt{2U(u^*)} \geq
-\sqrt{2c}\vert u^*-p\vert,\quad t\geq t_\rho.
\]
\end{proof}

We now show that if $\Gamma_+$ has positive diameter then $T_+<+\infty$. To prove this we first show that $T_+=+\infty$ implies $u^*(t)\rightarrow p\in P$ as $t\rightarrow+\infty$, then we conclude that this is in contrast with (\ref{IV0}).
\begin{lemma}\label{Convp}
If $T_+=+\infty$, then
there is $p\in P$ such that
\begin{equation}\label{LimP}
\lim_{t\rightarrow+\infty}u^*(t)=p.
\end{equation}
An analogous statement applies to $T_-$.
\end{lemma}
\begin{proof}
If $\Gamma_+=\{p\}$ for some $p\in P$,  then
(\ref{LimP}) follows by (\ref{LimD0}). Therefore we assume that
$\Gamma_+$ has positive diameter.  The idea of the proof is to show
that if $u^*(t)$ gets too close to $\partial\Gamma_+\setminus P$ it is
forced to end up on $\Gamma_+\setminus P$ in a finite time in
contradiction with $T^*=+\infty$.

If (\ref{LimP}) does not hold there is $q>0$ and a sequence
{$\{\tau_j\}$, with $\lim_{j\to\infty}\tau_j=+\infty$,}
such that $d(u^*(\tau_j),P)\geq q$,
for all $j\in\N$.
Since, by
(\ref{Bound}) $u^*$ is bounded, using also (\ref{LimD0}), we can
assume that
\begin{equation}\label{tauSequence}
\lim_{j\rightarrow+\infty}u^*(\tau_j) = \bar{x},\;\;\text{for
  some}\;\;\bar{x}\in\Gamma_+\setminus\cup_{p\in P} B_q(p).
\end{equation}
The smoothness of $U$ implies that there
 are positive constants $\bar{r}$, $r$, $c$ and $C$ such
  that
\begin{enumerate}
\item the orthogonal projection on
  $\pi:B_{\bar{r}}(\bar{x})\rightarrow\partial\Omega$ is well defined
  and $\pi(B_{\bar{r}}(\bar{x}))\subset\partial\Omega\setminus P$;
\item we have
\[
B_r(x_0)\subset B_{\bar{r}}(\bar{x}),\;\;\mbox{for all }
x_0\in\partial\Omega\cap B_{\frac{\bar{r}}{2}}(\bar{x});
\]
\item if $(\xi,s)\in\R^{n-1}\times\R$ are local coordinates with
  respect to a basis $\{e_1,\ldots,e_n\}$, $e_j=e_j(x_0)$, with
  $e_n(x_0)$ the unit interior normal to $\partial\Omega$ at
  $x_0\in\partial\Omega\cap B_{\frac{\bar{r}}{2}}(\bar{x})$ it results
  \begin{equation}\label{P2}
\frac{1}{2}c s\leq U(x(x_0,(\xi,s)))\leq 2c s,\;\;\vert
\xi\vert^2+s^2\leq r^2,\; s\geq h(x_0,\xi),
\end{equation}
where
\[
x=x(x_0,(\xi,s))=x_0+\sum_{j=1}^n \xi_je_j(x_0) + se_n(x_0),
\]
and $h:\partial\Omega\cap B_{\frac{\bar{r}}{2}}(\bar{x}) \times
\{\vert \xi\vert\leq r\}\rightarrow\R$, $\vert h(x_0,\xi)\vert\leq
C\vert \xi\vert^2$, for $\vert \xi\vert\leq r$, is a local
representation of $\partial\Omega$ in a neighborhood of $x_0$,
that is $U(x(x_0,(\xi,h(x_0,\xi))))=0$ for $\vert \xi\vert\leq r$.
\end{enumerate}
Fix a value $j_0$ of $j$ and set $t_0=\tau_{j_0}$.
%
\begin{figure}
\centerline{\includegraphics[width=10cm]{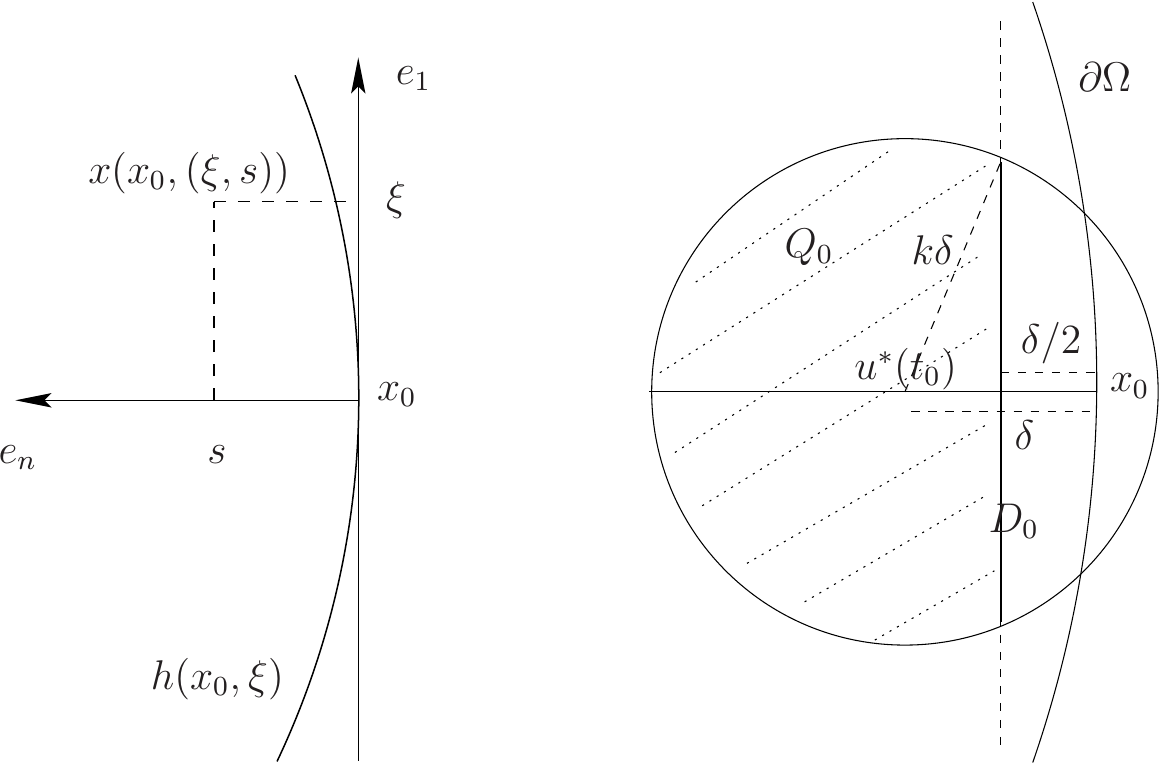}}
  \caption{The coordinates $(\xi,s)$ and the domain $Q_0$ in Lemma~\ref{Convp}.}
\end{figure}
If $j_0$ is sufficiently large, setting $t_0=\tau_{j_0}$ we
have that $x_0=\pi(u^*(t_0))$ is well defined. Moreover
$x_0\in\partial\Omega\cap B_{\frac{\bar{r}}{2}}(\bar{x})$ and
\[
u^*(t_0)=x_0+\delta e_n(x_0),\;\;\delta=\vert u^*(t_0)-x_0\vert.
\]
For $k=\frac{8}{3}\sqrt{2}$ let $Q_0$ be the set
\[
Q_0=\{x(x_0,(\xi,s)):\vert
\xi\vert^2+(s-\delta)^2<k^2\delta^2,\;s>{\delta}/{2}\}.
\]
Since $\delta\rightarrow 0$ as $j_0\rightarrow+\infty$ we can assume
that $\delta>0$ is so small ($\delta < \min\{\frac{1}{2Ck^2},
\frac{r}{1+k}\}$ suffices) that $\overline{Q}_0\subset\Omega\cap
B_r(x_0)$.

\vskip.2cm
\noindent
{\em Claim 1.} $u^*(t)$
leaves $\overline{Q}_0$ through the disc $D_0={\partial
Q_0\setminus\partial B_{k\delta}(u^*(t_0))}$.

\vskip.2cm
From (\ref{a0bound}) we have $a_0\leq\mathcal{A}(v,(T_-,T_+^v))$ for
each $W^{1,2}$ map ${v:(T_-,T_+^v]\rightarrow\R^n}$ that coincides
  with $u^*$ for $t\leq t_0$, and satisfies
  $v((t_0,T_+^v))\subset\Omega$, $v(T_+^v)\in\partial\Omega$ and
  (\ref{EnergConserv}). Therefore if we set
  \[
  w(s)=x_0+s e_n(x_0),
  \]
  $s\in[0,\delta]$, we have
\begin{equation}\label{Leftest}
a_0\leq \mathcal{A}(u^*,(T_-,t_0)) + \mathcal{J}(w,(0,\delta)).
\end{equation}
On the other hand, if $u^*(t_0^\prime)\in\partial Q_0(x_0)\cap\partial
B_{k\delta}(u^*(t_0))$, where
\[
t_0^\prime=\sup\{t>t_0:u^*([t_0,t))\subset\overline{Q}_0\setminus\partial
  B_{k\delta}(u^*(t_0))\},
\]
from (\ref{IV}) it follows
\begin{equation}\label{Rightest}
\mathcal{A}(u^*,(T_-,t_0))+\mathcal{J}(u^*,(t_0,t_0^\prime))\leq a_0.
\end{equation}
Using (\ref{P2}) we obtain
\begin{equation}\label{Upest}
\mathcal{J}(w,(0,\delta))\leq\frac{4}{3}c^\frac{1}{2}\delta^\frac{3}{2},
\end{equation}
and, since
\[
c\frac{\delta}{4}\leq
U(x(x_0,(\xi,s))),\quad (\xi,s)\in\overline{Q}_0(x_0),
\]
we also have, with $k$ defined above,
\begin{equation}\label{Downest}
\frac{8}{3}c^\frac{1}{2}\delta^\frac{3}{2} =
\frac{k}{\sqrt{2}}c^\frac{1}{2}\delta^\frac{3}{2}\leq
\frac{c^\frac{1}{2}\delta^\frac{1}{2}}{\sqrt{2}}
\int_{t_0}^{t_0^\prime}\vert\dot{u}^*(t)\vert dt \leq
\sqrt{2}\int_{t_0}^{t_0^\prime} \sqrt{U(u^*(t))}\vert\dot{u}^*(t)\vert
dt.
\end{equation}
From (\ref{Upest}) and (\ref{Downest}) it follows
\[
\mathcal{J}(w,(0,\delta)) \leq
\frac{1}{2}\mathcal{J}(u^*,(t_0,t_0^\prime)),
\]
and therefore (\ref{Leftest}) and (\ref{Rightest}) imply the absurd
inequality $a_0<a_0$. This contradiction proves the claim.

\vskip.2cm From Claim 1 it follows that there is $t_1\in(t_0,+\infty)$
with the following properties:
\[\begin{split}
&u^*([t_0,t_1))\subset Q_0(x_0),\\
&u(t_1)\in D_0.
\end{split}\]
Set $x_{0,1}=\pi(u^*(t_1))$ and $\delta_1=\vert
u^*(t_1)-x_{0,1}\vert$.
Since $h(x_0,0)=h_\xi(x_0,0)=0$ and the radius
$\rho_\delta=(k^2-\frac{1}{4})^{\frac{1}{2}}\delta$ of $D_0$ is
proportional to $\delta$, we can assume that $\delta$ is so small that
the ratio $\frac{2\delta_1}{\delta}$ and $\frac{\vert x_{0,1} -
  x_0\vert}{\vert u^*(t_1)-x(x_0,(0,\frac{\delta}{2}))\vert}$ are near
$1$ so that we have
\[
\begin{split}
&\delta_1\leq\rho\delta,\;\;\text{for some}\;\;\rho<1,\\
&\vert x_{0,1}-x_0\vert\leq k\delta.
\end{split}
\]


\noindent We also have
\[
t_1-t_0\leq k^\prime\delta^\frac{1}{2},\;\; k^\prime =
\frac{8k}{c^\frac{1}{2}}.
\]
This follows from
\[
\begin{split}
&(t_1-t_0)\frac{c}{4}\delta \leq\mathcal{A}(u^*,(t_0,t_1))
=\mathcal{J}(u^*,(t_0,t_1)) \\
&=\sqrt{2}\int_{t_0}^{t_1}\sqrt{U(u^*(t))}\vert\dot{u^*}(t)\vert dt
\leq 2\sqrt{c\delta}\vert u^*(t_1)-u^*(t_0)\vert \leq 2 c^{\frac{1}{2}} k \delta^{\frac{3}{2}}.
\end{split}
\]
where we used (\ref{P2}) to estimate $\mathcal{J}$ on the segment
joining $u^*(t_0)$ with $u^*(t_1)$.

\noindent We have $u^*(t_1)=x_{0,1}+\delta_1e_n(x_{0,1})$ and we can apply Claim
1 to deduce that there exists $t_2>t_1$ such that
\[
\begin{split}
&u^*([t_1,t_2))\subset Q_1(x_{0,1}),\\ &u^*(t_2)\in D_1,
\end{split}
\]
where $Q_1$ and $D_1$ are defined as $Q_0$ and $D_0$ with $\delta_1$
and $x(x_{0,1},(\xi,s))$ instead of $\delta$ and
$x(x_0,(\xi,s))$. Therefore an induction argument yields sequences
$\{t_j\}$, $\{x_{0,j}\}$, $\{\delta_j\}$ and $\{Q_j(x_{0,j})\}$ such
that
\begin{equation}\label{Sequence}
\begin{split}
&u^*([t_j,t_{j+1}))\subset Q_j(x_{0,j}), \quad x_{0,j}=\pi(u^*(t_j)),\\
&\delta_{j+1}\leq\rho\delta_j\leq\rho^{j+1}\delta,\\
&\vert x_{0,j+1}-x_{0,j}\vert\leq k\delta_j\leq k\rho^j\delta,\\
&(t_{j+1}-t_j)\leq k^\prime\delta_j^{1/2}\leq k^\prime\rho^{j/2}\delta^{1/2},\\
&u^*(t_j)=x_{0,j}+\delta_je_n(x_{0,j}) \in D_j.
\end{split}
\end{equation}
We can also assume that $Q_j(x_{0,j})\subset\Omega\cap B_r(x_0)$,
 for all $j\in\N$.
This follows from $\vert u^*(t_{j+1})-u^*(t_j)\vert\leq
k\delta_j\leq k\rho^j\delta$.

\noindent From (\ref{Sequence}) we obtain that
there exists $T$ with
$t_0<T\leq\frac{k^\prime\delta^\frac{1}{2}}{1-\rho^\frac{1}{2}}$ such
that
\[
\begin{split}
  &u^*(T)=\lim_{t\rightarrow T}u^*(t)=\lim_{j\rightarrow+\infty}
  x_{0,j}\in \partial\Omega\setminus P,\\
&\vert u^*(T)-x_0\vert\leq\frac{k\delta}{1-\rho}.
\end{split}
\]
This contradicts the existence of the sequence $\{\tau_j\}$, with $\lim_{j\to\infty}\tau_j=+\infty$, appearing in
(\ref{tauSequence}) and establishes (\ref{LimP}).  The proof
 of the lemma is complete.
\end{proof}

We continue by showing (\ref{LimP}) contradicts (\ref{IV0}).


\begin{lemma}\label{T<}
Assume that $\Gamma_+$ has positive diameter. Then
\[T_+<+\infty.\]
An analogous statement applies to $\Gamma_-$ and $T_-$.
\end{lemma}
\begin{proof}
From Lemma~\ref{Convp}, if $T_+=+\infty$ there exists $p\in P$ such
that $\lim_{t\to+\infty}u^*(t) = p$.  We use a local argument to show
that this is impossible if $\Gamma_+$ has positive diameter.
By a suitable change of variable we can
assume that $p=0$ and that, in a neighborhood of $0\in\R^n$, $U$ reads
\[U(u)=V(u)+W(u),\]
where $V$ is the quadratic part of $U$:
\begin{equation}\label{Quadratic}
V(u)=\frac{1}{2}\Bigl(-\sum_{i=1}^m\lambda_i^2u_i^2 +
\sum_{i=m+1}^n\lambda_i^2u_i^2\Bigr), \qquad \lambda_i>0
\end{equation} and $W$ satisfies,
\begin{equation}\label{W}
\vert W(u)\vert\leq C\vert u\vert^3, \quad \vert W_x(u)\vert\leq C\vert
u\vert^2, \quad \vert W_{xx}(u)\vert\leq C\vert u\vert.
\end{equation}
Consider the Hamiltonian system with
\[
H(p,q) = \frac{1}{2}\vert p\vert^2 - U(q),\quad p\in\R^n,\ 
q\in\Omega\subset\R^n.
\]
For this system the origin of $\R^{2n}$ is an
equilibrium point that corresponds to the critical point $p=0$ of
$U$.
Set $D = \text{diag}(-\lambda_1^2,
\ldots,-\lambda_m^2,\lambda_{m+1}^2, \ldots,\lambda_n^2)$.
The eigenvalues of the symplectic matrix
\[
\left(\begin{array}{cc}
0&D\\
I&0\end{array}\right)
\]
are
\[
\begin{split}
&-\lambda_i,\;\;i=m+1,\ldots,n\\
&\phantom{-\;\,}\lambda_i,\;\;i=m+1,\ldots,n\\
&\pm i\lambda_i,\;\;i=1,\ldots,m.
\end{split}
\]
Let
$(e_1,0),\ldots,(e_n,0)$, $(0,e_1),\ldots,(0,e_n)$ be the basis of
$\R^{2n}$ defined by $e_j=(\delta_{j1},\ldots,\delta_{jn})$, where
$\delta_{ji}$ is Kronecker's delta.
The stable $S^s$, unstable $S^u$ and center $S^c$ subspaces invariant
under the flow of the linearized Hamiltonian system at $0\in\R^{2n}$
are
\[
\begin{split}
&S^s=\text{span}\{(-\lambda_j e_j,e_j)\}_{j=m+1}^n,\\
&S^u=\text{span}\{(\lambda_j e_j,e_j)\}_{j=m+1}^n,\\
&S^c=\text{span}\{(e_j,0), (0,e_j)\}_{j=1}^m.
\end{split}
\]
From (\ref{LimP}) and (\ref{EnergConserv}) we have
\[
\lim_{t\rightarrow+\infty}(\dot{u}^*(t),u^*(t))=0\in\R^{2n}.
\]
Let $W^s$ and $W^u$  be the local stable and unstable
manifold and let $W^c$ be a local center
manifold at $0\in\R^{2n}$. From the center manifold theorem
\cite{B},
\cite{W}, there is a constant $\lambda_0>0$ such that, for each
solution $(p(t),q(t))$ that remains in a neighborhood of $0\in\R^{2n}$
for positive time, there is a solution $(p^c(t),q^c(t))\in W^c$ that
satisfies
\begin{equation}\label{LimH0}
\vert(p(t),q(t))-(p^c(t),q^c(t))\vert=\text{O}(e^{-\lambda_0t}).
\end{equation}
Since $W^c$ is tangent to $S^c$ at $0\in\R^{2n}$, the projection
$W_0^c$ on the configuration space is tangent to
$S_0^c=\text{span}\{e_j\}_{j=1}^m$, which is the projection of $S^c$ on the
configuration space. Therefore, if $(p^c,q^c) \not\equiv 0$, given
$\gamma>0$, by (\ref{LimH0}) there is $t_\gamma$ such that
$d(q(t),S_0^c)\leq\gamma\vert q(t)\vert$, for $t\geq t_\gamma$. For
$\gamma$ small, this implies that $q(t)\not\in\Omega$ for $t\geq
t_\gamma$. It follows that $(p^c,q^c)\equiv 0$ and from (\ref{LimH0})
$(p(t),q(t))$ converges to zero exponentially. This is possible only
if $(p(t),q(t))\in W^s$ and, in turn, only if $q(t)\in W_0^s$, the
projection of $W^s$ on the configuration space. This argument leads to
the conclusion that the trajectory of $u^*$ in a neighborhood of $0$
is of the form
\begin{equation}\label{uForm}
u^*(t(s))=\mathfrak{u}^*(s) = s\eta+z(s),
\end{equation}
where
\[
\eta=\sum_{i=m+1}^n\eta_ie_i
\]
is a unit vector\footnote{Actually
  $\eta$ coincides with one of the eigenvectors of $U''(0)$.},
$s\in[0,s_0)$ for some $s_0>0$, and $z(s)$ satisfies
\begin{equation}\label{zeta}
 z(s)\cdot\eta=0,\quad \vert z(s)\vert\leq c\vert s\vert^2,
\quad\vert z^\prime(s)\vert\leq
c\vert s\vert
\end{equation}
for a positive constant $c$.

We are now in the position of constructing our local perturbation of
$u$.  We first discuss the case $U=V$, $z(s)=0$.  We set
\[
\bar{u}(s)=s\eta
\]
and, in some interval
$[1,s_1]$,
construct a
competing map $\bar{v}:[1,s_1]\rightarrow\R^n$,
\[
\bar{v} = \bar{u} + g e_1,
\quad g:[1,s_1]\rightarrow\R,
\]
with the following
properties:
\begin{eqnarray}
&& V(\bar{v}(1))=0, \nonumber\\
&& \bar{v}(s_1)=\bar{u}(s_1),\nonumber\\
&&\mathcal{J}_V(\bar{v},[1,s_1])<\mathcal{J}_V(\bar{u},[0,s_1]).\label{iii}
\end{eqnarray}

The basic observation is that, if we move from $\bar{u}$ in the
direction of one of the eigenvectors $e_1,\ldots,e_m$ corresponding to
negative eigenvalues of the Hessian of $V$, the potential $V$
decreases and therefore, for each $s_0\in(1,s_1)$ we can define the
function $g$ in the interval $[1,s_0]$ so that
\begin{equation}\label{EqualJ}
    \mathcal{J}_V(\bar{u}+ge_1, (1,s_0))=\mathcal{J}_V(\bar{u},
          (1,s_0)).
\end{equation}
Indeed it suffices to impose that
$g:(1,s_0]\rightarrow\R$
  satisfies the condition
\[
  \sqrt{V(\bar{u}(s))} = \sqrt{1+g'^2(s)}
\sqrt{V(\bar{u}(s)+g(s)e_1)},\;\;s\in(1,s_0].
\]
According with this condition we take $g$  as the
solution of the problem
\begin{equation}\label{gEquat}
\left\{\begin{array}{l}{g}' =
\displaystyle -\frac{\lambda_1g}{\sqrt{s^2\lambda_\eta^2-\lambda_1^2g^2}} =
-\frac{\frac{\lambda_1g}{s\lambda_\eta}}{\sqrt{1 -
    \frac{\lambda_1^2g^2}{s^2\lambda_\eta^2}}}\\
g(1)=\frac{\lambda_\eta}{\lambda_1}\end{array}\right.,
\end{equation}
where we have used (\ref{Quadratic}) and set
\[
\lambda_\eta=\sqrt{\sum_{i=m+1}^n\lambda_i^2\eta_i^2}.
\]
Note that the initial condition in (\ref{gEquat}) implies
$V(\bar{v}(1))=0$. The solution $g$ of (\ref{gEquat}) is well defined
in spite of the fact that the right hand side tends to
$-\infty$ as $s\rightarrow 1$.
Since $g$ defined by (\ref{gEquat}) is positive for $s\in[1,+\infty)$,
to satisfy the condition $\bar{v}(s_1)=\bar{u}(s_1)$, we give a
suitable definition of $g$ in the interval $[s_0,s_1]$ in order that
$g(s_1)=0$. Choose a number $\alpha\in(0,1)$ and extend $g$ with
continuity to the interval $[s_0,s_1]$ by imposing that
\begin{equation}\label{alphaCond}
\sqrt{V(\bar{u}(s))} = \alpha\sqrt{1 + g'^2(s)}
\sqrt{V(\bar{u}(s)+g(s)e_1)},\;\;s\in(s_0,s_1].
\end{equation}

Therefore, in the interval $(s_0,s_1]$, we define $g$ by
\begin{equation}\label{gEquat1}
{g}' = -\frac{1}{\alpha}\sqrt{\frac{1 - \alpha^2 +
    \alpha^2\frac{\lambda_1^2g^2}{s^2\lambda_\eta^2}} {1 -
    \frac{\lambda_1^2g^2}{s^2\lambda_\eta^2}}}\leq -
\frac{\sqrt{1-\alpha^2}}{\alpha}.
\end{equation}

\begin{figure}
%
\centerline{\includegraphics[width=6.6cm]{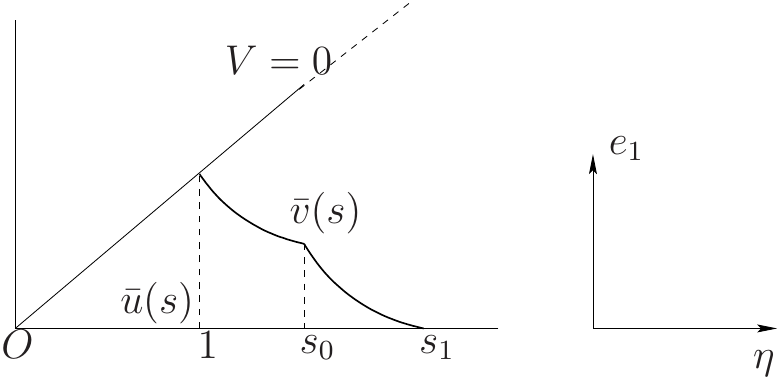}}
\caption{The maps $\bar{u}(s)$ and $\bar{v}(s)$.}
\end{figure}

\noindent Since (\ref{alphaCond}) implies
\[
\mathcal{J}_V(\bar{v},[s_0,s_1]) =
\frac{1}{\alpha}\mathcal{J}_V(\bar{u},[s_0,s_1]),
\]
from (\ref{EqualJ}) we see that $\bar{v}$ satisfies also the
requirement (\ref{iii}) above if we can choose $\alpha\in(0,1)$ and
$1<s_0<s_1$ in such a way that
\[
\mathcal{J}_V(\bar{u},(0,1)) >
\frac{1-\alpha}{\alpha}\mathcal{J}_V(\bar{u},(s_0,s_1)).
\]
Since (\ref{gEquat1}) implies $s_1<s_0+\frac{\alpha
  g(s_0)}{\sqrt{1-\alpha^2}}$ a sufficient condition for this is
\[
\mathcal{J}_V(\bar{u},(0,1)) >
\frac{1-\alpha}{\alpha}\mathcal{J}_V\Bigl(\bar{u},
     \Bigl(s_0,s_0+\frac{\alpha g(s_0)}{\sqrt{1-\alpha^2}}\Bigr)\Bigr),
\]
or equivalently
\begin{equation}
1 > \frac{1-\alpha}{\alpha}\Big(\Bigl(s_0+\frac{\alpha
  g(s_0)}{\sqrt{1-\alpha^2}}\Bigr)^2 - s_0^2\Big) =
2s_0g(s_0)\sqrt{\frac{1-\alpha}{1+\alpha}} + \frac{\alpha
  g^2(s_0)}{1+\alpha}.
\label{ineqeq}
\end{equation}
By a proper choice of $s_0$ and $\alpha$ the right hand side of (\ref{ineqeq})
can be made as small as we like. For instance we can
fix $s_0$ so that $g(s_0)\leq\frac{1}{4}$ and then choose
$\alpha$ in such a way that
$\frac{1}{2}s_0\sqrt{\frac{1-\alpha}{1+\alpha}}\leq\frac{1}{4}$ and
conclude that (\ref{iii}) holds.

\smallbreak
Next we use the function $g$
 to define a comparison map $v$ that coincides with
$u^*$ outside an $\epsilon$-neighborhood of $0$ and show that the
assumption that the trajectory of $u^*$ ends up in some $p\in P$ must
be rejected. For small $\epsilon>0$ we define
\begin{equation}\label{vDef}
v(\epsilon s) = \epsilon s\eta+z(\epsilon s) + \epsilon
g(s-\sigma)e_1,\;\;s\in[1+\sigma,s_1+\sigma],
\end{equation}
where $\sigma=\sigma(\epsilon)$ is determined by the condition
\[
U(v(\epsilon(1+\sigma)))=0,
\]
which, using (\ref{Quadratic}), (\ref{W}), (\ref{zeta}) and
$g(1)=\frac{\lambda_\eta}{\lambda_1}$, after dividing by $\epsilon^2$,
becomes
\begin{equation}\label{sigma}
\frac{1}{2}\lambda_\eta^2((1+\sigma)^2-1)=\epsilon f(\sigma,\epsilon),
\end{equation}
where $f(\sigma,\epsilon)$ is a smooth bounded function defined in a
neighborhood of $(0,0)$. For small $\epsilon>0$, there is a unique
solution $\sigma(\epsilon)=\text{O}(\epsilon)$ of (\ref{sigma}). Note
also that (\ref{vDef}) implies that
\[
v(\epsilon(s_1+\sigma))=\mathfrak{u}^*(\epsilon(s_1+\sigma)).
\]
We now conclude by showing that, for $\epsilon>0$ small, it results
\begin{equation}\label{Jineq}
\mathcal{J}_U(\mathfrak{u}^*(\epsilon\cdot),(0,s_1+\sigma)) >
\mathcal{J}_U(v(\epsilon\cdot),(1+\sigma,s_1+\sigma)).
\end{equation}
From (\ref{uForm}) and (\ref{vDef}) we have
\begin{equation}\label{Der}
\lim_{\epsilon\rightarrow 0^+}\epsilon^{-1}
\Bigl\vert\frac{d}{ds}\mathfrak{u}^*(\epsilon s)\Bigr\vert
=1,\;\;\;
\lim_{\epsilon\rightarrow 0^+}\epsilon^{-1}
\Bigl\vert\frac{d}{ds}v(\epsilon s)\Bigr\vert
= \sqrt{1+g'^2(s)},
\end{equation}
and, using also (\ref{W}) and $\sigma=\text{O}(\epsilon)$,

\begin{equation}\label{Ulim}
\begin{split}
&\lim_{\epsilon\rightarrow 0^+}\epsilon^{-2}U(\mathfrak{u}^*(\epsilon s)) =
  V(\bar{u}(s)),\;\;s\in(0,s_1),\\
&\lim_{\epsilon\rightarrow 0^+}\epsilon^{-2}U(v(\epsilon s)) =
  V(\bar{v}(s)),\;\;s\in(1,s_1)
\end{split}
\end{equation}
uniformly in compact intervals.

The limits (\ref{Der}) and (\ref{Ulim}) imply
\[
\begin{split}
&\lim_{\epsilon\rightarrow
    0^+}\epsilon^{-2}\mathcal{J}_U(\mathfrak{u}^*(\epsilon\cdot),
  (0,s_1+\sigma)) = \lim_{\epsilon\rightarrow
    0^+}\sqrt{2}\int_0^{s_1+\sigma}
  \sqrt{\epsilon^{-2}U(\mathfrak{u}^*(\epsilon s))}\epsilon^{-1}
  \Bigl\vert\frac{d}{ds}\mathfrak{u}^*(\epsilon s)\Bigr\vert
  ds,\\
&=\sqrt{2}\int_0^{s_1} \sqrt{V(\bar{u}(s))}
  ds = \mathcal{J}_V(\bar{u},(0,s_1))\\
&\lim_{\epsilon\rightarrow
    0^+}\epsilon^{-2}\mathcal{J}_U(v(\epsilon\cdot),(1+\sigma,s_1+\sigma))
  =\lim_{\epsilon\rightarrow
    0^+}\sqrt{2}\int_{1+\sigma}^{s_1+\sigma}
  \sqrt{\epsilon^{-2}U(v(\epsilon s))}\epsilon^{-1}
  \Bigl\vert\frac{d}{ds}v(\epsilon s)\Bigr\vert ds,\\
&=\sqrt{2}\int_1^{s_1}
  \sqrt{V(\bar{v}(s))}\sqrt{1+g'^2(s)} ds =
  \mathcal{J}_V(\bar{v},(1,s_1)).
\end{split}
\]
This and (iii) above imply that, indeed, the inequality (\ref{Jineq})
holds for small $\epsilon>0$. The proof is complete.
\end{proof}
\noindent
We can now complete the proof of Theorem \ref{main0}. We show that the
map $u^*:(T_-,T_+)\rightarrow\R^n$ possesses all the required
properties. The fact that $u^*$ satisfies (\ref{Newton}) and
(\ref{EnergConserv}) follows from Lemma \ref{FirstInt}. Lemma
\ref{timelemma} implies (\ref{d-0}) and, if $T_->-\infty$, also
(\ref{x-}). The fact that $x_-\in\Gamma_-\setminus P$ is a consequence
of Lemma \ref{limd0} and implies that $\Gamma_-$ has positive
diameter. Viceversa, if $\Gamma_-$ has positive diameter, Lemmas
\ref{Convp} and \ref{T<} imply that $T_->-\infty$ and that (\ref{x-})
holds for some $x_-\in\Gamma_-\setminus P$. The proof of Theorem
\ref{main0} is complete.

\begin{remark}
From Theorem~\ref{main0} it follows that if $N$ is even then there are
at least $N/2$ distinct orbits connecting different elements of
$\{\Gamma_1,\ldots,\Gamma_N\}$.  If $N$ is odd there are at least
$(N+1)/2$.
Simple examples show that, given distinct $\Gamma_i, \Gamma_j\in
\{\Gamma_1,\ldots,\Gamma_N\}$, an orbit
connecting them does not always exist.
Let
\[
  \mathcal{U}_{ij} = \{ u\in
  W^{1,2}((T_-^u,T_+^u);\R^n):u((T_-^u,T_+^u))\subset\Omega,
    u(T_-^u)\in\Gamma_i, u(T_+^u)\in\Gamma_j
  \}
\]
with $i\neq j$ and
\[
d_{ij} = \inf_{u\in\mathcal{U}_{ij}}\mathcal{A}(u,(T_-^u,T_+^u)).
\]
An orbit connecting $\Gamma_i$ and $\Gamma_j$ exists if
\[
d_{ij} < d_{ik} + d_{kj}, \quad \forall k\neq i,j.
\]
\end{remark}

\smallbreak The proof of Theorem \ref{main-0} uses, with obvious
modifications, the same arguments as in
the proof of Theorem \ref{main0} to characterize $u^*$ as the limit of
a minimizing sequence $\{u_j\}$ of the action functional
\begin{equation}
\mathcal{A}(u,(0,T^u)) = \int_0^{T^u} \Bigl(\frac{1}{2}|\dot{u}(t)|^2
+ U(u(t))\Bigr)dt.
\end{equation}
in the set
\begin{equation}\label{SymmU}
\mathcal{U} = \{ u\in W^{1,2}((0,T^u);\R^n) :
0<T_+^u <+\infty,\ u(0) = 0, \ u([0,T_+^u))\subset\Omega, \ u(T_+^u)
  \in \partial\Omega \}.
\end{equation}

\begin{remark}\label{tempidiversi}
%
  In the symmetric case of Theorem~\ref{main-0} it is easy to
  construct an example with $T_+<T_+^\infty$.  For $U(x)=1-\vert
  x\vert^2$, $x\in\R^2$, the solution $u:[0,{\pi}/{2}]\rightarrow\R^2$
  of (\ref{Newton}) determined by (\ref{EnergConserv}) and
  $u([0,{\pi}/{2}])=\{(s,0):s\in[0,1]\}$ is a minimizer of
  $\mathcal{A}$ in $\mathcal{U}$. For $\eps$ small, let
  $t_\eps=\arcsin(1-\eps)$ and define $u_\epsilon: [0,T^{u_\epsilon}]
  \rightarrow \R^2$ as the map determined by (\ref{EnergConserv}),
  $u_\epsilon([0,t_\epsilon]) = \{(s,0): s\in[0,1-\eps)\}$ and
    $u_\eps((t_\epsilon, T^{u_\eps}]) =
  \{(1-\eps,s):s\in(0,\sqrt{2\eps-\eps^2}]\}$. In this case
$T_+={\pi}/{2}$ and $T_+^\infty=3\pi/4$.
\end{remark}

\subsection{On the existence of heteroclinic connections}
\label{s:heteroclinic}

Corollary \ref{Hetero} states the existence of heteroclinic
connections under the assumptions of Theorem~\ref{main0} and, in
particular, that $U\in C^2$.  Actually, by examining the proof of
Theorem~\ref{main0} we can establish an existence result under weaker
hypotheses. In the special case $\partial\Omega=P$, $\#P\geq 2$,
given $p_-\in P$, the set $\mathcal{U}$ defined in (\ref{Ad}) takes
the form
\[
\begin{split}
\mathcal{U}=\bigl\{u\in W^{1,2}((T_-^u,T_+^u);\R^n):
        {-\infty< T_-^u < T_+^u< +\infty,}\hskip 2.2cm\\
        u((T_-^u,T_+^u))\subset\Omega,\ U(u(0))=U_0,\
        u(T_-^u)=p_-,\
        u(T_+^u)\in P\setminus\{p_-\}\}.
\end{split}
\]
In this section we slightly enlarge the set $\mathcal{U}$ by allowing
$T_\pm^u=\pm\infty$ and consider the admissible set
\[
\begin{split}
\widetilde{\mathcal{U}}=\bigl\{u\in W^{1,2}_{loc}((T_-^u,T_+^u);\R^n):
        {-\infty\leq T_-^u < T_+^u\leq+\infty,}\hskip 2.2cm\\
        u((T_-^u,T_+^u))\subset\Omega,\ U(u(0))=U_0,\
        \lim_{t\rightarrow T_-^u}u(t)=p_-,\
        \lim_{t\rightarrow T_+^u}u(t)\in P\setminus\{p_-\}\}.
\end{split}
\]

\begin{proposition}\label{prosant}
Assume that $U$ is a non-negative continuous function, which vanishes in a
finite set $P$, $\#P\geq 2$, and satisfies
\[
\sqrt{U(x)}\geq\sigma(\vert x\vert),\;\; x\in\Omega,\;\;
  \vert x\vert\geq r_0
\]
for some $r_0>0$ and
a non-negative function $\sigma:[r_0,+\infty)\rightarrow\R$ such that
$\int_{r_0}^{+\infty}\sigma(r)dr=+\infty$.
 
Given $p_-\in P$ there is $p_+\in P\setminus\{p_-\}$ and a
  Lipschitz-continuous map $u^*:(T_-,T_+)\rightarrow\Omega$ that satisfies
  \eqref{EnergConserv} almost everywhere on $(T_-,T_+)$, 
\[
\lim_{t\rightarrow T_\pm}u^*(t)=p_\pm,
\]
and minimizes the action functional $\mathcal{A}$ on $\tilde{\mathcal{U}}$.
\end{proposition}

\begin{proof}
We begin by showing that
\begin{equation}\label{=}
a_0=\inf_{u\in\mathcal{U}}\mathcal{A}=\inf_{u\in\tilde{\mathcal{U}}}\mathcal{A}=\tilde{a}_0.
\end{equation}
Since $\mathcal{U}\subset\tilde{\mathcal{U}}$ we have
$a_0\geq\tilde{a}_0$. On the other hand arguing as in the proof of
Lemma \ref{timelemma}, if $T_+-T_-=+\infty$, given a small number
$\epsilon>0$, we can construct a map $u_\epsilon\in\mathcal{U}$ that
satisfies
\[
a_0\leq\mathcal{A}(u_\epsilon,(T_-^{u_\epsilon},T_+^{u_\epsilon}))\leq\mathcal{A}(u,(T_-^u,T_+^u))
+\eta_\epsilon
\]
where $\eta_\epsilon\rightarrow 0$ as $\epsilon\rightarrow 0$. This
implies $a_0\leq\tilde{a}_0$ and establishes (\ref{=}). It follows
that we can proceed as in the proof of Theorem \ref{main0} and define
$u^*\in\tilde{\mathcal{U}}$ as the limit of a minimizing sequence
$\{u_j\}\subset\mathcal{U}$. The arguments in the proof of Lemma
\ref{timelemma} show that (\ref{IV0}) holds. It remain to show that
$u^*$ is Lipschitz-continuous. Looking at the proof of Lemma
\ref{FirstInt} we see that the continuity of $U$ is sufficient for
establishing that (\ref{EnergConserv}) holds almost everywhere on
$(T_-,T_+)$, and the Lipschitz character of $u^*$ follows. The proof
is complete.
\end{proof}
\begin{remark}
Without further information on the behavior of $U$ in a neighborhood
of $p_\pm$ nothing can be said on $T_\pm$ being finite or infinite and
it is easy to construct examples to show that all possible
combinations are possible. As shown in Lemma \ref{limd0} a sufficient
condition for $T_\pm=\pm\infty$ is that, in a neighborhood of
$p=p_\pm$, $U(x)$ is bounded by a function of the form $c\vert
x-p\vert^2$, $c>0$. $U$ of class $C^1$ is a sufficient condition in
order that $u^*$ is of class $C^2$ and satisfies (\ref{Newton}).
\end{remark}


\section{Examples}
In this section we show a few simple applications of Theorems
\ref{main0} and \ref{main-0}.

\noindent Our first application describes a class of potentials with
the property that, in spite of the existence of possibly infinitely
many critical values, (\ref{Newton}) has a nontrivial periodic orbit
on any energy level.
\begin{proposition}\label{allEnerg}
Assume that $U:\R^n\rightarrow\R$ satisfies
\[\begin{split}
&U(-x)=U(x),\;\;x\in\R^n,\\
&U(0) = 0, \ U(x)< 0\ \mbox{for}\ x\neq 0,\\
&\lim_{|x|\to\infty}U(x) = -\infty
\end{split}\]
Assume moreover that each non zero critical point
  of $U$ is hyperbolic with Morse index $i_m\geq 1$. Then there is a
  nontrivial periodic orbit of (\ref{Newton}) on the energy level
  $\frac{1}{2}\vert\dot{u}\vert^2-U(u)=\alpha$ for each $\alpha>0$.
\end{proposition}
\begin{proof}
For each $\alpha>0$ we set $\tilde{U}=U(x)+\alpha$ and let
$\Omega\subset\{\tilde{U}>0\}$ be the connected component that contains
the origin. $\Omega$ is open, nonempty and bounded and, from the
assumptions on the properties of the critical points of $U$, it
follows that $\partial\Omega$ is connected and contains at most a
finite number of critical points. Therefore we are under the
assumptions of Corollary \ref{caseN-1} for the case $N=1$ and the
existence of the periodic orbit follows.
\end{proof}
An example of potential $U:\R^2\rightarrow\R$ that satisfies the
assumptions in Proposition \ref{allEnerg} is, in polar coordinates
$r,\theta$,
\[
U(r,\theta) = -r^2+\frac{1}{2}\tanh^4(r)
\cos^2(r^{-1})\cos^{2k}(2\theta),
\]
where $k>0$ is a sufficiently large number.

\begin{figure}[t!]
%
\centerline{\includegraphics[width=8cm]{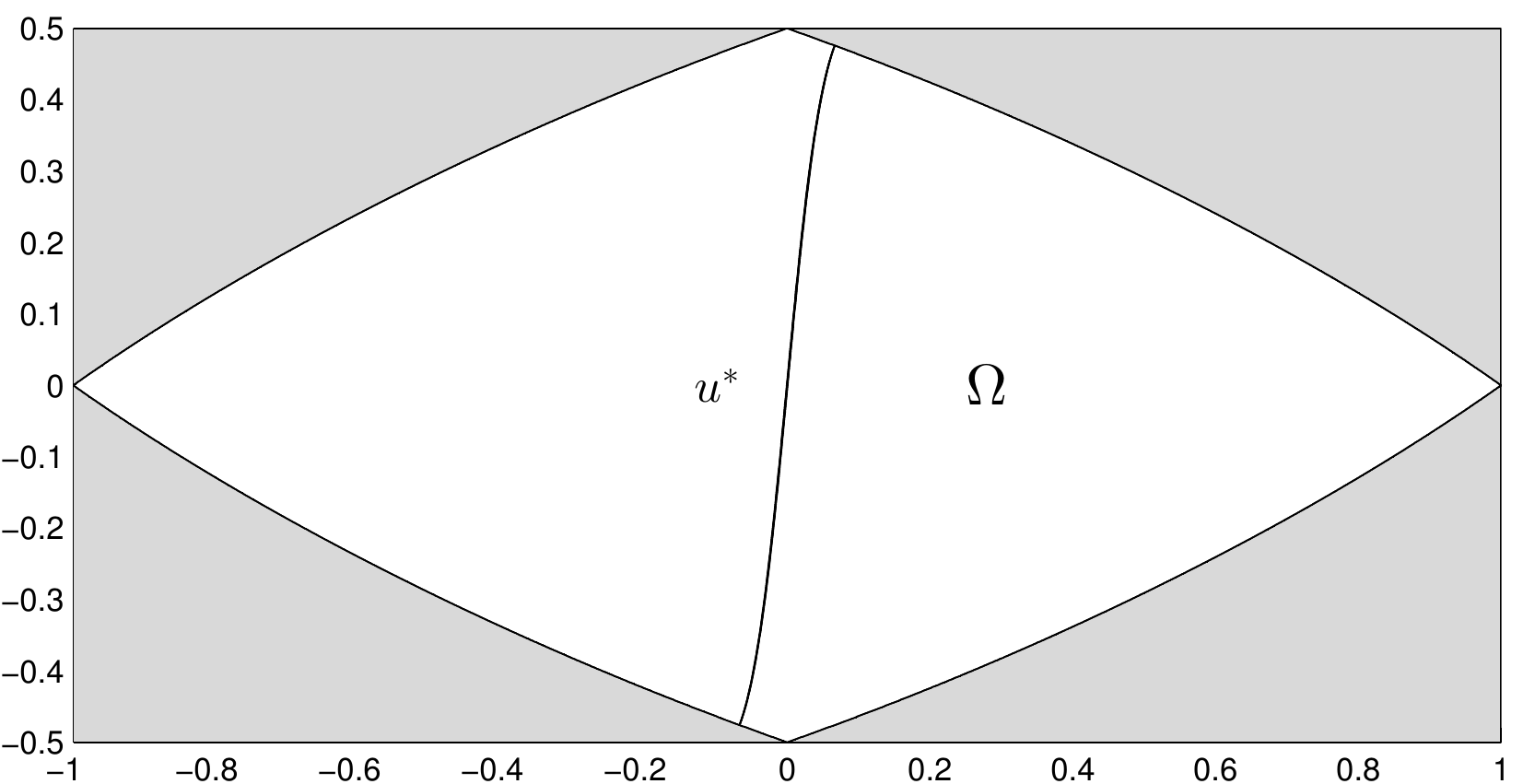}}
\caption{Symmetric periodic orbit for the example with potential (\ref{ex2}).}
\end{figure}

\smallskip
Next we give another application of Corollary \ref{caseN-1}.  For the
potential $U:\R^2\rightarrow\R$, with
\begin{equation}
U(x)=\frac{1}{2}(1-x_1^2)^2+\frac{1}{2}(1-4x_2^2)^2,
\label{ex2}
\end{equation}
the energy level
 $\alpha=-\frac{1}{2}$
is critical and corresponds to
four hyperbolic critical points $p_1=(1,0)$, $-p_1$,
$p_2=(0,\frac{1}{2})$ and $-p_2$. The connected component
$\Omega\subset\{\tilde{U}>0\}$, ($\tilde{U}=U(x)-\frac{1}{2}$) that
contains the origin is bounded by a simple curve $\Gamma$ that
contains $\pm p_1$ and $\pm p_2$. In spite of the presence of these
critical points, from Theorem \ref{main-0} it follows that there is a
minimizer $u\in\mathcal{U}$, with $\mathcal{U}$ as in
(\ref{SymmU}) and $u(T^u)\in\Gamma\setminus\{\pm p_1,\pm p_2\}$, and
Corollary \ref{caseN-1} implies the existence of a periodic solution $v^*$.
Note that there are also two heteroclinic orbits, solutions of
(\ref{Newton}) and (\ref{EnergConserv}):
\[
u_1(t) = (\tanh(t),0),\qquad  u_2(t) = (0,\frac{1}{2}\tanh(2t)).
\]
These orbits connect $p_j$ to $-p_j$, for $j=1,2$. By
Theorem~\ref{main-0} both $u_1$ and $u_2$ have action greater than
$v^*|_{(-T_+,T_+)}$.

\begin{figure}[t!]
\centerline{
\includegraphics[width=7.8cm]{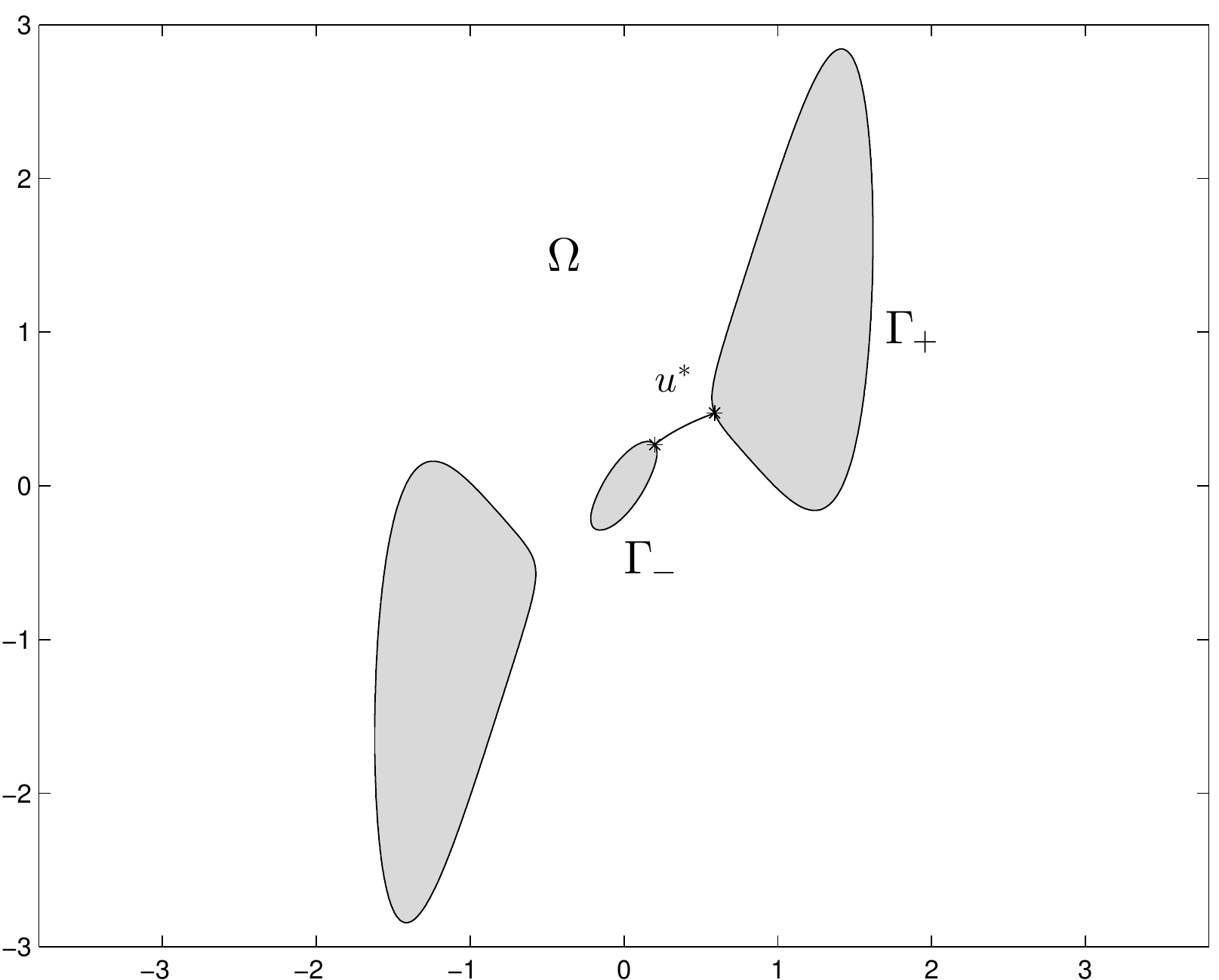}
\includegraphics[width=7.8cm]{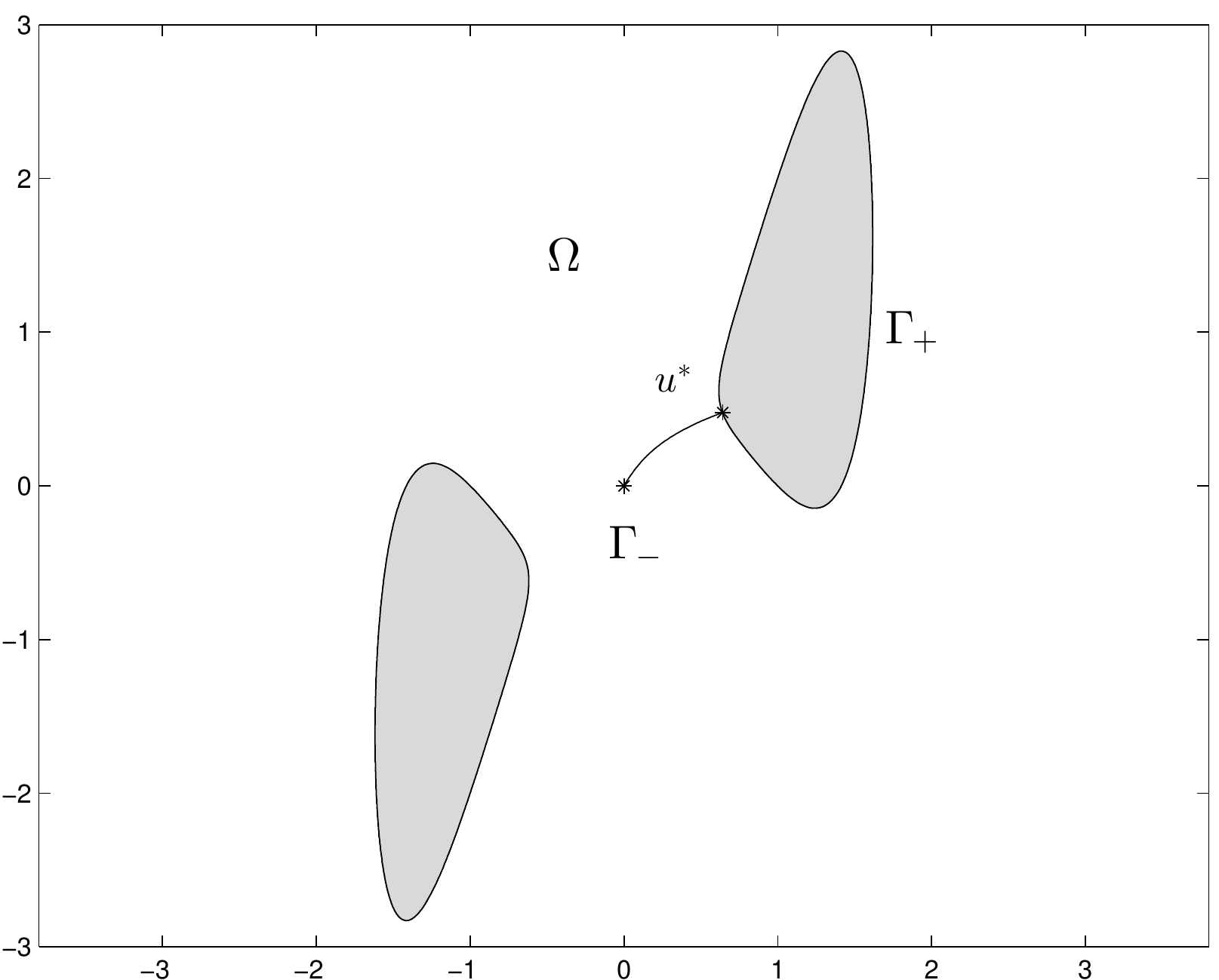}
}
\centerline{
\includegraphics[width=7.8cm]{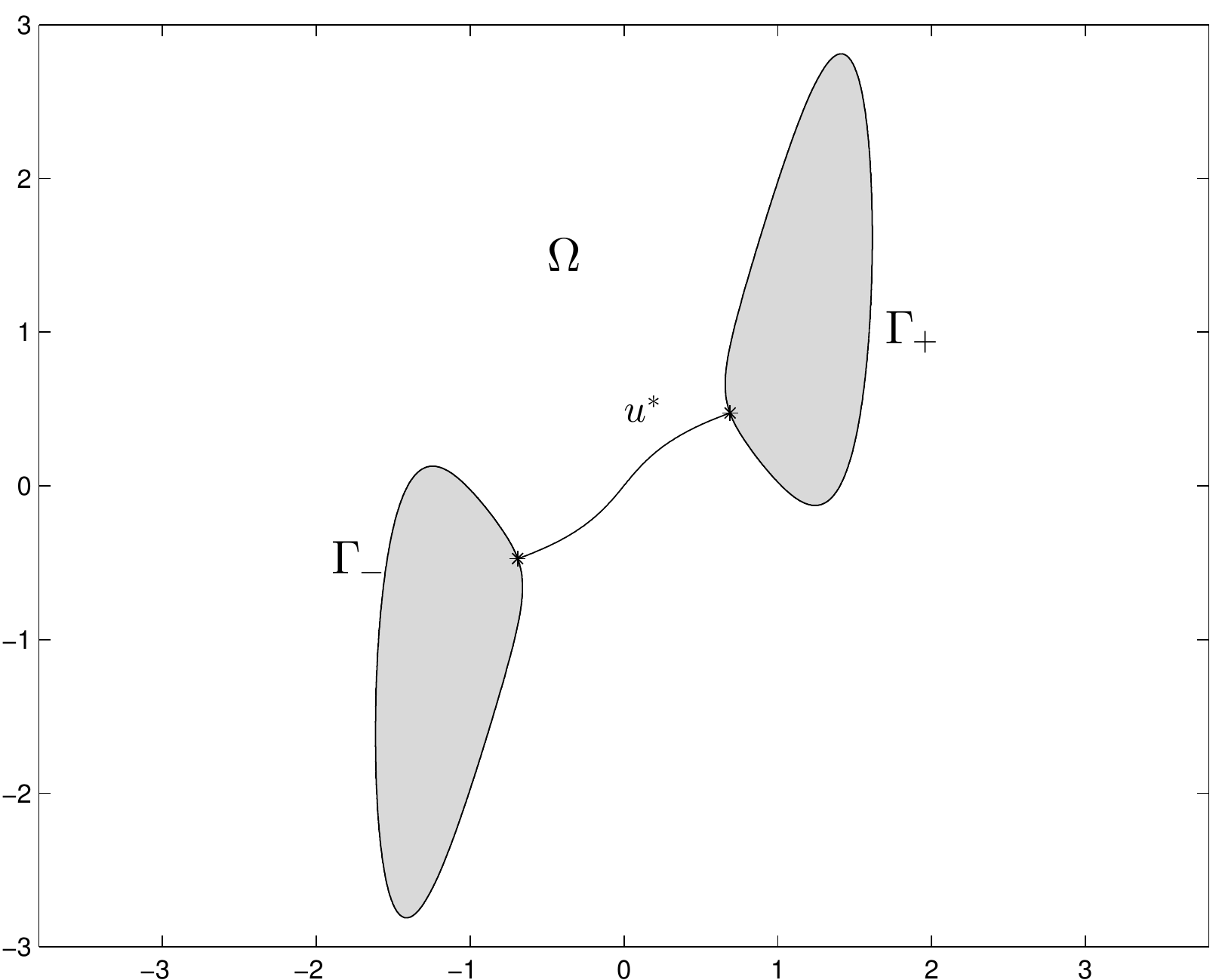}
\includegraphics[width=7.8cm]{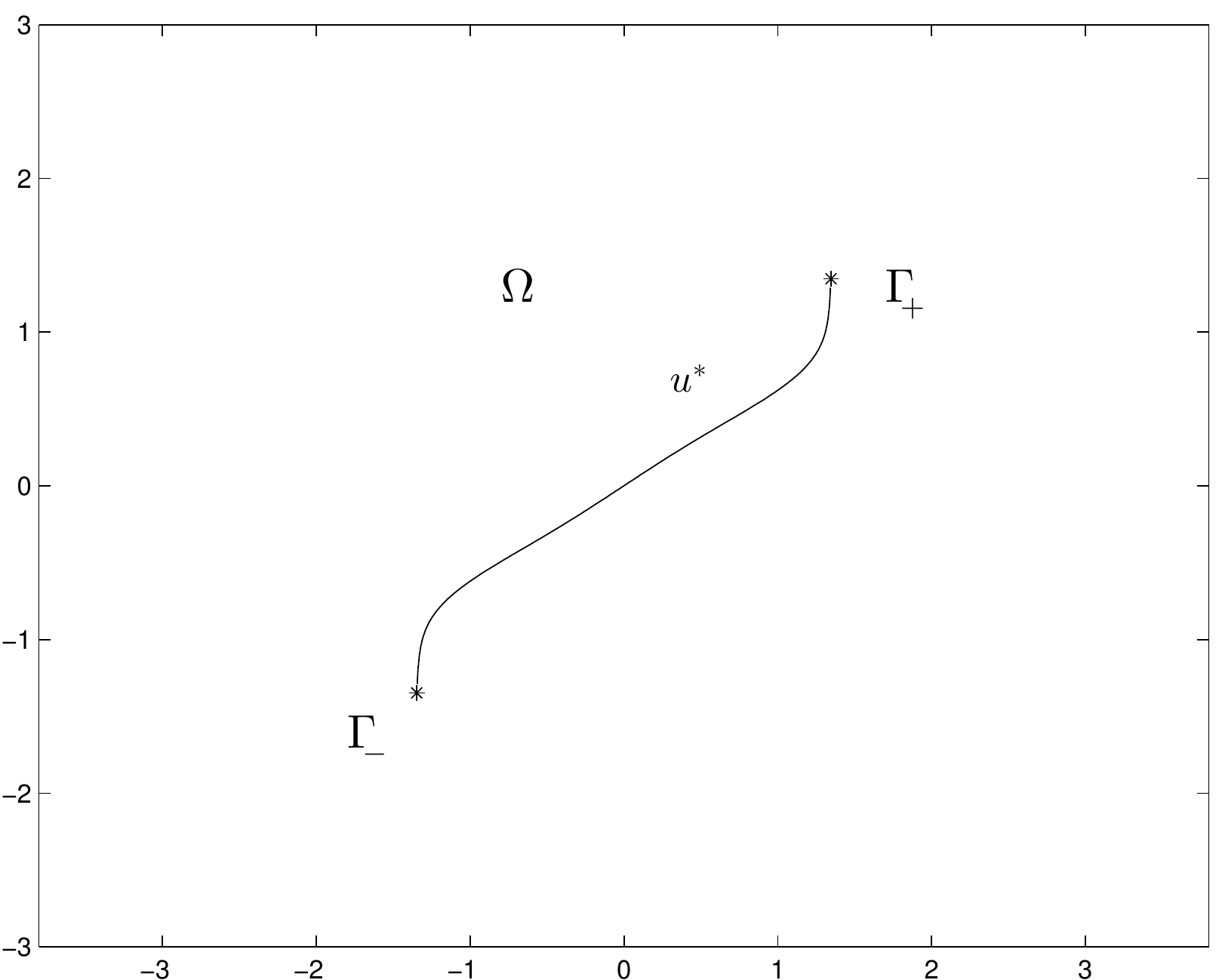}
}
\caption{Bifurcations of dynamics of (\ref{Newton}) with the
$\alpha=0$, bottom left: $\alpha=0.05$, bottom right: $\alpha
=-U(p_2,1)$. The shaded regions are not accessible.}
\end{figure}

\smallskip
Our last example shows
that Theorems \ref{main0} and \ref{main-0} can be used to derive
information on the rich dynamics that (\ref{Newton}) can exhibit when
$U$ undergoes a small perturbation. We consider a family of potentials
$U:\R^2\times[0,1]\rightarrow\R$. We assume that $U(x,0)=x_1^6+x_2^2$
which from various points of view is a structurally unstable potential
and, for $\lambda>0$ small, we consider the perturbed potential
\begin{equation}
U(x,\lambda)=2\lambda^4x_1^2+x_2^2-2\lambda^2x_1x_2-3\lambda^2x_1^4+x_1^6.
\label{Uesempio2}
\end{equation}
This potential satisfies $U(-x,\lambda)=U(x,\lambda)$ and, for
$\lambda>0$, has the five critical points $p_0, \pm p_1$
and $\pm p_2$ defined by
\begin{eqnarray*}
&p_0=(0,0),\\
&p_1 = (\lambda(1-(\frac{2}{3})^\frac{1}{2})^\frac{1}{2},
  \lambda^3(1-(\frac{2}{3})^\frac{1}{2})^\frac{1}{2}),\\
&p_2 = (\lambda(1+(\frac{2}{3})^\frac{1}{2})^\frac{1}{2},
  \lambda^3(1+(\frac{2}{3})^\frac{1}{2})^\frac{1}{2}),
\end{eqnarray*}
which are all hyperbolic.

We have $U(p_2,\lambda)<0=U(p_0,\lambda)<U(p_1,\lambda)$ and $p_0$ is
a local minimum, $p_1$ a saddle and $p_2$ a global
minimum. Let $\alpha$ be the energy level. For
$-\alpha<U(p_2,\lambda)$ or $-\alpha\geq U(p_1,\lambda)$ no
information can be derived from Theorems \ref{main0} and \ref{main-0}
therefore we assume $-\alpha\in[U(p_2,\lambda),U(p_1,\lambda))$. For
  $-\alpha=U(p_2,\lambda)$ Corollary \ref{Hetero} or Corollary
  \ref{caseN-1}
 yields the existence of a heteroclinic connection $u_2$
between $-p_2$ and $p_2$. For $-\alpha\in(U(p_2,\lambda),0)$ Corollary
\ref{caseN-1} implies the existence of a periodic orbit
$u_\alpha$. This periodic orbit converges uniformly in compact intervals to
$u_2$ and the period $T_\alpha\rightarrow+\infty$ as
$-\alpha\rightarrow U(p_2,\lambda)^+$. For $\alpha=0$ Corollary
\ref{Homo} implies the existence of two orbits $u_0$ and $-u_0$
homoclinic to $p_0=0$. We can assume that $u_0$ satisfies the
condition $u_0(-t)=u_0(t)$ and that $u_\alpha(0)=0$. Then we have that
$u_\alpha(\cdot\pm\frac{T_\alpha}{4})$ converges uniformly in compact intervals
to $\mp u_0$ and $T_\alpha\rightarrow+\infty$ as $-\alpha\rightarrow
0^-$. For $-\alpha\in(0,U(p_1,\lambda))$, $\partial\Omega$ is the
union of three simple curves all of positive diameter: $\Gamma_0$ that
includes the origin and $\pm\Gamma_2$ which includes $\pm p_2$ and
Corollary \ref{Per} together with the fact that $U(\cdot,\lambda)$ is
symmetric imply the existence of two periodic solutions
$\tilde{u}_\alpha$ and $-\tilde{u}_\alpha$ with $\tilde{u}_\alpha$
that oscillates between $\Gamma_0$ and $\Gamma_2$ in each time
interval equal to $\frac{T_\alpha}{2}$. Assuming that
$\tilde{u}_\alpha(0)\in\Gamma_2$ we have that, as $-\alpha\rightarrow
0^+$, $\tilde{u}_\alpha\rightarrow u_0$ uniformly in compacts and
$T_\alpha\rightarrow+\infty$. Finally we observe that, in the limit
$-\alpha\rightarrow U(p_1,\lambda)^-$, $\tilde{u}_\alpha$ converges
uniformly in $\R$ to the constant solution $u\equiv p_1$.

\subsection*{Acknowledgements}
The first author is indebted with Peter Bates for fruitful discussions on
the subject of this paper.

\bibliographystyle{plain}

\begin{thebibliography}{99}

\bibitem{af} N.~Alikakos and G.~Fusco.  \newblock On the connection
  problem for potentials with several global minima.  \newblock {\em
    Indiana\ Univ.\ Math.\ Journ.} {\bf 57} No.~4, 1871-1906 (2008)

\bibitem{A} P.~Antonopoulos and P.~Smyrnelis.  \newblock On minimizers
  of the Hamiltonian system $u''=\nabla W(u)$, and on the existence of
  heteroclinic, homoclinic and periodic connections.  \newblock
  Preprint (2016)

\bibitem{braides} A.~Braides. \newblock Approximation of
  Free-Discontinuity Problems. Lectures Notes in Mathematics 1694,
  Springer-Verlag, Heidelberg (1998)

\bibitem{B} A.~Bressan.  \newblock Tutorial on the Center Manifold
  Theorem. Hyperbolic systems of balance laws. CIME course (Cetraro
  2003). Springer Lecture Notes in Mathematics 1911,
  327-344. Springer-Verlag, Heidelberg (2007)

\bibitem{BGH} G. Buttazzo, M. Giaquinta, S. Hildebrandt. \newblock
  One-dimensional Calculus of Variations: an Introduction \newblock
  Oxford University Press, Oxford (1998)


\bibitem{salomao} N. V. De Paulo and P. A. S. Salom$\tilde{\mathrm{a}}$o.
  \newblock Systems of transversal sections near critical energy
  levels of Hamiltonian systems in $\R^4$, arXiv:1310.8464v2 (2016)

\bibitem{monteil} A.~Monteil and F.~Santambrogio.  \newblock Metric
  methods for heteroclinic connections. {\em Mathematical Methods in the
  Applied Sciences}, \newblock DOI: 10.1002/mma.4072 (2016)

\bibitem{sourdis} C.~Sourdis.  \newblock The heteroclinic connection
  problem for general double-well potentials.
\newblock {\em Mediterranean Journal of Mathematics} {\bf 13} No. 6,
  4693-4710 (2016)

\bibitem{sternberg} P. Sternberg. \newblock Vector-Valued Local
  Minimizers of Nonconvex Variational Problems, {\em Rocky Mountain
  J. Math.} {\bf 21} No.~2, 799-807 (1991)

\bibitem{W} A.~Vanderbauwhede.  \newblock Centre manifolds, normal
  forms and elementary bifurcations.  \newblock {\em Dynamics
    Reported} {\bf 2} No.~4, 89-169 (1989)

\bibitem{ZS} A.~Zuniga and P.~Sternberg.  \newblock On the
  heteroclinic connection problem for multi-well gradient systems.
  \newblock {\em Journal of Differential Equations} {\bf 261} No.~7,
  3987-4007 (2016)

\end{thebibliography}

\end{document}